\documentclass[a4paper]{amsart}

\usepackage{color}
\usepackage[latin1]{inputenc}
\usepackage[T1]{fontenc}
\usepackage{amsfonts}
\usepackage{amssymb}
\usepackage{hyperref}
\usepackage{amsmath}
\usepackage{amsthm}
\usepackage{stmaryrd}
\usepackage{enumerate}
\usepackage{multirow}
\usepackage{graphicx}
\usepackage{xcolor}
\usepackage{ulem}
\usepackage{comment}
\newtheorem{theorem}{Theorem}[section]
\newtheorem{lemma}[theorem]{Lemma}
\newtheorem{proposition}[theorem]{Proposition}

\newtheorem{assumption}[theorem]{Assumption}
\newtheorem{remark}[theorem]{Remark}
\begin{document}
\setlength\arraycolsep{2pt}
\title[Tight Risk Bound for High Dimensional Time Series Completion]{Tight Risk Bound for High Dimensional Time Series Completion}
\author{Pierre ALQUIER}
\address{*RIKEN AIP, Tokyo, Japan}
\email{pierre.alquier.stat@gmail.com}
\author{Nicolas MARIE$^{\dag}$}
\address{$^{\dag,\diamond}$Laboratoire Modal'X, Universit\'e Paris Nanterre, Nanterre, France}
\email{nmarie@parisnanterre.fr}
\author{Am\'elie ROSIER$^{\diamond}$}
\address{$^{\diamond}$ESME Sudria, Paris, France}
\email{amelie.rosier@esme.fr}
\date{}
\maketitle
\noindent
%


%
\begin{abstract}
Initially designed for independent datas, low-rank matrix completion was successfully applied in many domains to the reconstruction of partially observed high-dimensional time series. However, there is a lack of theory to support the application of these methods to dependent datas. In this paper, we propose a general model for multivariate, partially observed time series. We show that the least-square method with a rank penalty leads to reconstruction error of the same order as for independent datas. Moreover, when the time series has some additional properties such as periodicity or smoothness, the rate can actually be faster than in the independent case.
\end{abstract}
\tableofcontents
%


%
\section{Introduction}

Low-rank matrix completion methods were studied in depth in the past 10 years. This was partly motivated by the popularity of the Netflix prize~\cite{BL07} in the machine learning community. The first theoretical papers on the topic covered matrix recovery from a few entries observed exactly~\cite{CR09,CT10,keshavan}. The same problem was studied with noisy observations in~\cite{CP10,CP11,keshavan2,gross}. The minimax rate of estimation was derived by~\cite{KLT11}. Since then, many estimators and many variants of this problem were studied in the statistical literature, see~\cite{nega,klopp,lafond,klopp2,mai,suzuki,cottet,CKLN18,AR20,mai0,mai1} for instance.
\\
High-dimensional time series often have strong correlation, and it is thus natural to assume that the matrix that contains such a series is low-rank (exactly, or approximately). Many econometrics models are designed to generate series with such a structure. For example, the factor model studied in~\cite{koop,lam,lam2,giordani,chan,hallin} can be interpreted as a high-dimensional autoregressive (AR) process with a low-rank transition matrix. This model (and variants) was used and studied in signal processing~\cite{BLM19} and statistics~\cite{nega,ABDG20}. Other papers focused on a simpler model where the series is represented by a deterministic low-rank trend matrix plus some possibly correlated noise. This model was used by~\cite{yu} to perform prediction, and studied in~\cite{AM19}.
\\
It is thus tempting to use low-rank matrix completion algorithms to recover partially observed high-dimensional time series, and this was indeed done in many applications:~\cite{xie,tsag,eshke} used low-rank matrix completion to reconstruct data from multiple sensors. Similar techniques were used by~\cite{edf,edf2} to recover the electricity consumption of many households from partial observations, by~\cite{ABDIK18} on panel data in economics, and by~\cite{poulos,BN19} for policy evaluation. Some algorithms were proposed to take into account the temporal updates of the observations (see \cite{shi}). However, it is important to note that 1) all the aforementioned theory on matrix completion, for example~\cite{KLT11}, was only developed for independent observations, and 2) most papers using these techniques on time series did not provide any theoretical justification that it can be used on dependent observations. One must however mention that~\cite{gillard} obtained theoretical results for univariate time series prediction by embedding the time series into a Hankel matrix and using low-rank matrix completion.
\\
In this paper, we study low-rank matrix completion for partially observed high-dimensional time series that indeed exhibit a temporal dependence. We provide a risk bound for the reconstruction of a rank-$k$ matrix, and a model selection procedure for the case where the rank $k$ is unknown. Under the assumption that the univariate series are $\phi$-mixing, we prove that we can reconstruct the matrix with a similar error than in the i.i.d case in~\cite{KLT11}. If, moreover, the time series has some additional properties, as the ones studied in~\cite{AM19} (periodicity or smoothness), the error can even be smaller than in the i.i.d case. This is confirmed by a short simulation study.
\\
From a technical point of view, we start by a reduction of the matrix completion problem to a structured regression problem as in~\cite{mai}. But on the contrary to~\cite{mai}, we have here dependent observations. We thus follow the technique of~\cite{ALW13} to obtain risk bounds for dependent observations. In~\cite{ALW13}, it is shown that one can obtain risk bounds for dependent observations that are similar to the risk bounds for independent observations under a $\phi$-mixing assumption, using Samson's version of Bernstein inequality~\cite{SAMSON00}. For model selection, we follow the guidelines of~\cite{massart}: we introduce a penalty proportional to the rank. Using the previous risk bounds, we show that this leads to an optimal rank selection.
\\
The paper is organized as follows. In Section~\ref{section_setting}, we introduce our model, and the notations used throughout the paper. In Section~\ref{section_risk_bounds}, we provide the risk analysis when the rank $k$ is known. We then describe our rank selection procedure in Section~\ref{section_model_selection} and show that it satisfies a sharp oracle inequality. The numerical experiments are in Section~\ref{section_experiments}. All the proofs are gathered in Section~\ref{section_proofs}.
\\
\\
\textbf{Notations and basic definitions.} Throughout the paper, $\mathcal M_{d,T}(\mathbb R)$ is equipped with the Fr\"ob\'enius scalar product
\begin{displaymath}
\langle .,.\rangle_{\mathcal F} : (\mathbf A,\mathbf B)\in\mathcal M_{d,T}(\mathbb R)^2
\longmapsto
\textrm{trace}(\mathbf A^*\mathbf B) =
\sum_{j,t}\mathbf A_{j,t}\mathbf B_{j,t}
\end{displaymath}
or with the spectral norm
\begin{displaymath}
\|.\|_{\textrm{op}} :\mathbf A\in\mathcal M_{d,T}(\mathbb R)
\longmapsto
\sup_{\|x\| = 1}\|\mathbf Ax\| =\sigma_1(\mathbf A).
\end{displaymath}
Let us finally remind the definition of the $\phi$-mixing condition on stochastic processes. Given two $\sigma$-algebras $\mathcal A$ and $\mathcal B$, we define the $\phi$-mixing coefficient between $\mathcal A$ and $\mathcal B$ by
\begin{displaymath}
\phi(\mathcal A,\mathcal B) :=
\sup\left\{|\mathbb P(B) -\mathbb P(B|A)|
\textrm{ $;$ }(A,B)\in\mathcal A\times\mathcal B\textrm{, }\mathbb P(A)\neq 0\right\}.
\end{displaymath}
When $\mathcal{A}$ and $\mathcal{B}$ are independent, $\phi(\mathcal A,\mathcal B)=0$, more generally, this coefficient measure how dependent $\mathcal{A}$ and $\mathcal{B}$ are. Given a process $(Z_t)_{t\in\mathbb N}$, we define its $\phi$-mixing coefficients by
\begin{displaymath}
\phi_Z(i) :=\sup\left\{\phi(A,B)
\textrm{ $;$ }t\in\mathbb Z\textrm{, }A\in\sigma(X_h,h\leqslant t)\textrm{, }B\in\sigma(X_\ell,\ell\geqslant t + i)\right\}.
\end{displaymath}
Some properties and examples of applications of $\phi$-mixing coefficients can be found in~\cite{doukhan-mixing}.
%


%
\section{Setting of the problem and notations}\label{section_setting}
Consider $d,T\in\mathbb N^*$ and a $d\times T$ random matrix $\mathbf M$. Assume that the rows $\mathbf M_{1,.},\dots,\mathbf M_{d,.}$ are time series and that $Y_1,\dots,Y_n$ are $n\in\{1,\dots,d\times T\}$ noisy entries of the matrix $\mathbf M$:
\begin{equation}\label{noisy_model}
Y_i =\textrm{trace}(\mathbf X_{i}^{*}\mathbf M) +\xi_i
\textrm{ $;$ }i\in\{1,\dots,n\},
\end{equation}
where $\mathbf X_1,\dots,\mathbf X_n$ are i.i.d random matrices distributed on
\begin{displaymath}
\mathcal X :=
\{e_{\mathbb R^d}(j)e_{\mathbb R^T}(t)^*\textrm{ $;$ }
1\leqslant j\leqslant d\textrm{ and }
1\leqslant t\leqslant T\},
\end{displaymath}
and $\xi_1,\dots,\xi_n$ are i.i.d. centered random variables, with standard deviation $\sigma_{\xi} > 0$, such that $\mathbf X_i$ and $\xi_i$ are independent for every $i\in\{1,\dots,n\}$. Note that, as $\mathbf X_1,\dots,\mathbf X_n$ are independent, we do not exclude multiple observations of the same entry. That is, our model of matrix completion is the one studied in~\cite{KLT11,nega,mai} rather than the model in~\cite{CP10,CP11} where this is not possible.
\\
Let us now describe the time series structure of each  $\mathbf M_{1,.},\dots,\mathbf M_{d,.}$. We assume that each series $\mathbf M_{j,.}$ can be decomposed as a deterministic component $\mathbf\Theta^0_{j,.}$ plus some random noise $\varepsilon_{j,.}$. The noise can exhibit some temporal dependence: $\varepsilon_{j,t}$ will not be independent from $\varepsilon_{j,t'}$ in general. Moreover, as discussed in~\cite{AM19}, $\mathbf\Theta_{j,.}^{0}$ can have some more structure: $\mathbf\Theta_{j,.}^{0} =\mathbf T_{j,.}^{0}\mathbf\Lambda$ for some known matrix ${\bf \Lambda}$. Examples of such structures (smoothness or periodicity) are discussed below. This gives
\begin{equation}\label{ts_model}
\left\{
\begin{array}{rcl}
 \mathbf M & = &
 \mathbf\Theta^0 +\varepsilon\\
 \mathbf\Theta^0 & = &
 \mathbf T^0\mathbf\Lambda
\end{array}
\right.,
\end{equation}
where $\varepsilon$ is a $d\times T$ random matrix having i.i.d. and centered rows, $\mathbf\Lambda\in\mathcal M_{\tau,T}(\mathbb C)$ ($\tau\leqslant T$) is known and $\mathbf T^0$ is an unknown element of $\mathcal M_{d,\tau}(\mathbb R)$ such that
\begin{equation}\label{condition_trend}
\sup_{j,t}|\mathbf T_{j,t}^{0}|
\leqslant
\frac{\mathfrak m_0}{\mathfrak m_{\mathbf\Lambda}(\tau)}
\textrm{ with }
\mathfrak m_0 > 0
\textrm{ and }
1\vee\sup_{{\bf T}\in\mathcal M_{d,\tau}(\mathbb R)}\left\{
\frac{\sup_{j,t}|(\mathbf T\mathbf\Lambda)_{j,t}|}{\sup_{j,\ell}|\mathbf T_{j,\ell}|}
\right\}
\leqslant\mathfrak m_{\mathbf\Lambda}(\tau) <\infty.
\end{equation}
Note that this leads to
\begin{displaymath}
\sup_{j,t}|\mathbf\Theta_{j,t}^{0}|\leqslant
\sup_{j,\ell}|\mathbf T_{j,\ell}^{0}|\cdot
\frac{\sup_{j,t}|(\mathbf T^0\mathbf\Lambda)_{j,t}|}{\sup_{j,\ell}|\mathbf T_{j,\ell}^{0}|}
\leqslant\mathfrak m_0
\end{displaymath}
and
\begin{displaymath}
\mathfrak m_{\bf\Lambda} :=
\sup_{j,t}|\mathbf\Lambda_{j,t}| <\infty.
\end{displaymath}
We now make the additional assumption that the deterministic component is low-rank, reflecting the strong correlation between the different series. Precisely, we assume that $\mathbf T^0$ is of rank $k\in\{1,\dots,d\wedge T\}$: $\mathbf T^0 =\mathbf U^0\mathbf V^0$ with $\mathbf U^0\in\mathcal M_{d,k}(\mathbb R)$ and $\mathbf V^0\in\mathcal M_{k,\tau}(\mathbb R)$. The rows of the matrix $\mathbf V^0$ may be understood as latent factors. By Equations (\ref{noisy_model}) and (\ref{ts_model}), for any $i\in\{1,\dots,n\}$,
\begin{equation}\label{ts_model_bruite}
Y_i =\textrm{trace}(\mathbf X_{i}^{*}\mathbf\Theta^0) +
\overline\xi_i
\end{equation}
with $\overline\xi_i :=\textrm{trace}(\mathbf X_{i}^{*}\varepsilon) +\xi_i$. It is reasonable to assume that $\mathbf X_i$ and $\xi_i$, which are random terms related to the observation instrument, are independent to $\varepsilon$, which is the stochastic component of the observed process. Then, since $\xi_i$ is a centered random variable and $\varepsilon$ is a centered random matrix,
\begin{displaymath}
\mathbb E(\overline\xi_i) =
\mathbb E(\langle\mathbf X_i,\varepsilon\rangle_{\mathcal F}) +\mathbb E(\xi_i) =
\sum_{j = 1}^{d}
\sum_{t = 1}^{T}
\mathbb E((\mathbf X_i)_{j,t})\mathbb E(\varepsilon_{j,t}) = 0.
\end{displaymath}
This legitimates to consider the following least-square estimator of the matrix $\mathbf\Theta^0$:
\begin{equation}\label{estimator_definition}
\left\{
\begin{array}{rcl}
 \widehat{\mathbf\Theta}_{k,\tau} & = & \widehat{\mathbf T}_{k,\tau}\mathbf\Lambda\\
 \widehat{\mathbf T}_{k,\tau} & \in & \displaystyle{\arg\min_{\mathbf T\in\mathcal S_{k,\tau}}
 r_n(\mathbf T\mathbf\Lambda)}
\end{array}\right.,
\end{equation}
where $\mathcal S_{k,\tau}$ is a subset of
\begin{eqnarray*}
 \mathcal M_{d,k,\tau} & := &
 \left\{\mathbf U\mathbf V
 \textrm{ $;$ }
 (\mathbf U,\mathbf V)\in\mathcal M_{d,k}(\mathbb R)\times
 \mathcal M_{k,\tau}(\mathbb R)
 \textrm{ s.t. }
 \sup_{j,\ell}
 |\mathbf U_{j,\ell}|\leqslant\sqrt{\frac{\mathfrak m_0}
 {k\mathfrak m_{\mathbf\Lambda}(\tau)}}\right.\\
 & &
 \hspace{9cm}
 \left.\textrm{ and }
 \sup_{\ell,t}|\mathbf V_{\ell,t}|\leqslant\sqrt{\frac{\mathfrak m_0}
 {k\mathfrak m_{\mathbf\Lambda}(\tau)}}\right\},
\end{eqnarray*}
and
\begin{displaymath}
r_n(\mathbf A) :=
\frac{1}{n}
\sum_{i = 1}^{n}
(Y_i -\langle\mathbf X_i,\mathbf A\rangle_{\mathcal F})^2
\textrm{ $;$ }
\forall\mathbf A
\in\mathcal M_{d,T}(\mathbb R).
\end{displaymath}
%


%
\begin{remark}\label{matrices_subsets}
In many cases, we will simply take $\mathcal S_{k,\tau} =\mathcal M_{d,k,\tau}$. However, in many applications, it is natural to impose stronger constraints on the estimators. For example, in nonnegative matrix factorization, we would have
\begin{displaymath}
\mathcal S_{k,\tau} =
\{\mathbf U\mathbf V\textrm{ $;$ }(\mathbf U,\mathbf V)\in\mathcal M_{d,k,\tau}\textrm{ s.t. }\forall j,\ell,t\textrm{, }\mathbf U_{j,\ell}\geqslant 0\textrm{ and }\mathbf V_{\ell,t}\geqslant 0\}
\end{displaymath}
(see e.g.~\cite{edf}). So for now, we only assume that $\mathcal S_{k,\tau}\subset\mathcal M_{d,k,\tau}$. Later, we will specify some sets $\mathcal S_{k,\tau}$.
\end{remark}
\noindent
Let us conclude this section with two examples of matrices $\mathbf\Lambda$ corresponding to usual time series structures. On the one hand, if the trend of the multivalued time series $\mathbf M$ is $\tau$-periodic, with $T\in\tau\mathbb N^*$, one can take $\mathbf\Lambda = (\mathbf I_{\tau}|\cdots|\mathbf I_{\tau})$, and then $\mathfrak m_{\bf\Lambda} = 1$ and $\mathfrak m_{\bf\Lambda}(\tau) := 1$ works. So, in this case, note that the usual matrix completion model of \cite{KLT11} is part of our framework by taking $T =\tau$. On the other hand, assume that for any $j\in\{1,\dots,d\}$, the trend of $\mathbf M_{j,.}$ is a sample on $\{0,1/T,2/T,\dots,1\}$ of a function $f_j : [0,1]\rightarrow\mathbb R$ belonging to a Hilbert space $\mathcal H$. In this case, if $(\mathbf e_n)_{n\in\mathbb Z}$ is a Hilbert basis of $\mathcal H$, one can take $\mathbf\Lambda = (\mathbf e_n(t/T))_{|n|\leqslant N,1\leqslant t\leqslant T}$. For instance, if $f_j\in\mathbb L^2([0,1];\mathbb R)$, a natural choice is the Fourier basis $\mathbf e_n(t) = e^{2i\pi nt/T}$, and then $\mathfrak m_{\bf\Lambda} = 1$ and
\begin{displaymath}
\frac{\sup_{j,t}|(\mathbf T\mathbf\Lambda)_{j,t}|}{\sup_{j,\ell}|\mathbf T_{j,\ell}|}
\leqslant
\frac{1}{\sup_{j,\ell}|\mathbf T_{j,\ell}|}\cdot
\sup_{j,t}\sum_{\ell = 1}^{\tau}|\mathbf T_{j,\ell}e^{2i\pi nt/T}|
\leqslant\tau =:\mathfrak m_{\bf\Lambda}(\tau).
\end{displaymath}
Here, the usual matrix completion model of \cite{KLT11} is not part of our framework because $T$ is possibly huge and to take $\tau = T$ implies that the coefficients of the matrix $\mathbf T^0$ are all unrealistically small by Condition (\ref{condition_trend}). However, wathever the time series structure taken into account via $\mathbf\Lambda$, our model is designed for small values of $\tau$. Else, the model of \cite{KLT11} is appropriate. So, when $\mathbf\Lambda$ is the previous {\it Fourier matrix}, and in general when $\mathfrak m_{\bf\Lambda}(\tau)$ is a non constant increasing function of $\tau$, we assume that $\tau\in\llbracket 1,\tau_0\rrbracket$ with $\tau_0\ll T$.
%


%
\section{Risk bound on $\widehat{\mathbf T}_{k,\tau}$}
\label{section_risk_bounds}
%


%
\subsection{Upper bound}
First of all, since $\mathbf X_1,\dots,\mathbf X_n$ are i.i.d $\mathcal X$-valued random matrices, there exists a probability measure $\Pi$ on $\mathcal X$ such that
\begin{displaymath}
\mathbb P_{\mathbf X_i} =
\Pi\textrm{ $;$ }
\forall i\in\{1,\dots,n\}.
\end{displaymath}
In addition to the two norms on $\mathcal M_{d,T}(\mathbb R)$ introduced above, let us consider the scalar product $\langle .,.\rangle_{\mathcal F,\Pi}$ defined on $\mathcal M_{d,T}(\mathbb R)$ by
\begin{displaymath}
\langle\mathbf A,\mathbf B\rangle_{\mathcal F,\Pi} :=
\int_{\mathcal M_{d,T}(\mathbb R)}\langle X,\mathbf A\rangle_{\mathcal F}
\langle X,\mathbf B\rangle_{\mathcal F}\Pi(dX)
\textrm{ $;$ }
\forall\mathbf A,\mathbf B\in\mathcal M_{d,T}(\mathbb R).
\end{displaymath}
\textbf{Remarks:}
\begin{enumerate}
 \item For any deterministic $d\times T$ matrices $\mathbf A$ and $\mathbf B$,
 \begin{displaymath}
 \langle\mathbf A,\mathbf B\rangle_{\mathcal F,\Pi} =
 \mathbb E(\langle\mathbf A,\mathbf B\rangle_n)
 \end{displaymath}
 where $\langle .,.\rangle_n$ is the empirical scalar product on $\mathcal M_{d,T}(\mathbb R)$ defined by
 \begin{displaymath}
 \langle\mathbf A,\mathbf B\rangle_n :=
 \frac{1}{n}\sum_{i = 1}^{n}
 \langle\mathbf X_i,\mathbf A\rangle_{\mathcal F}
 \langle\mathbf X_i,\mathbf B\rangle_{\mathcal F}.
 \end{displaymath}
 However, note that this relationship between $\langle .,.\rangle_{\mathcal F,\Pi}$ and $\langle .,.\rangle_n$ doesn't hold anymore when $\mathbf A$ and $\mathbf B$ are random matrices.
 \item Note that if the sampling distribution $\Pi$ is uniform, then $\|.\|_{\mathcal F,\Pi}^{2} = (dT)^{-1}\|.\|_{\mathcal F}^{2}$.
\end{enumerate}
\noindent
\textbf{Notation.} For every $i\in\{1,\dots,n\}$, let $\chi_i$ be the couple of \textit{coordinates} of the nonzero element of $\mathbf X_i$, which is a $\mathcal E$-valued random variable with $\mathcal E =\{1,\dots,d\}\times\{1,\dots,T\}$.
\\
\\
In the sequel, $\varepsilon$, $\xi_1,\dots,\xi_n$ and $\mathbf X_1,\dots,\mathbf X_n$ fulfill the following additional conditions.
%


%
\begin{assumption}\label{assumption_phi_mixing}
The rows of $\varepsilon$ are independent and identically distributed. There is a process $(\varepsilon_t)_{t\in\mathbb Z}$ such that each $\varepsilon_{j,.}$ has the same distribution than $(\varepsilon_1,\dots,\varepsilon_T)$, and such that
\begin{displaymath}
\Phi_{\varepsilon} := 1 +\sum_{i = 1}^{n}\phi_{\varepsilon}(i)^{1/2} <\infty.
\end{displaymath}
\end{assumption}
%


%
\begin{assumption}\label{assumption_boundedness}
There exists a deterministic constant $\mathfrak m_{\varepsilon} > 0$ such that
\begin{displaymath}
\sup_{j,t}|\varepsilon_{j,t}|\leqslant
\mathfrak m_{\varepsilon}.
\end{displaymath}
Moreover, there exist two deterministic constants $\mathfrak c_{\xi},\mathfrak v_{\xi} > 0$ such that
\begin{displaymath}
\sup_{i\in\{1,\dots,n\}}\mathbb E(\xi_{i}^{2})\leqslant\mathfrak v_{\xi}
\end{displaymath}
and, for every $q\geqslant 3$,
\begin{displaymath}
\sup_{i\in\{1,\dots,n\}}\mathbb E(|\xi_i|^q)\leqslant
\frac{\mathfrak v_{\xi}\mathfrak c_{\xi}^{q - 2}q!}{2}.
\end{displaymath}
\end{assumption}
\noindent
This assumption means that the $\varepsilon_{j,t}$'s are bounded, and that the $\xi_i$'s are sub-exponential random variables. Sub-exponential random variables include bounded and Gaussian variables as special cases. Note that this is the assumption made on the noise for the matrix completion in the i.i.d. framework in the papers mentioned above~\cite{mai,KLT11}. The boundedness of the $\varepsilon_{j,t}$'s can be seen as quite restrictive. However, we are not aware of any way to avoid this assumption in this setting. Indeed, it allows to apply Samson's concentration inequality for $\phi$-mixing processes (see Samson \cite{SAMSON00}). In \cite{ALW13}, the authors prove sharp sparsity inequalities under a similar assumption, using Samson's inequality. They also show that the other concentration inequalities known for time series lead to slow rates of convergence.
%


%
\begin{assumption}\label{assumption_probabilities}
There is a constant $\mathfrak c_{\Pi} > 0$ such that
\begin{displaymath}
\Pi(\{e_{\mathbb{R}^d}(j)e_{\mathbb{R}^T}(t)^*\})
\leqslant
\frac{\mathfrak c_{\Pi}}{dT}
\textrm{ $;$ }\forall (j,t)\in\mathcal E.
\end{displaymath}
\end{assumption}
\noindent
Note that when the sampling distribution $\Pi$ is uniform, Assumption \ref{assumption_probabilities} is trivially satisfied with $\mathfrak c_{\Pi} = 1$.
%


%
\begin{theorem}\label{explicit_risk_bound}
Let $\alpha\in(0,1)$. Under Assumptions \ref{assumption_phi_mixing}, \ref{assumption_boundedness} and \ref{assumption_probabilities}, if $n\geqslant \max(d,\tau)$, then
\begin{displaymath}
\|\widehat{\mathbf\Theta}_{k,\tau} -\mathbf\Theta^0\|_{\mathcal F,\Pi}^{2}
\leqslant
3\min_{\mathbf T\in\mathcal S_{k,\tau}}
\|(\mathbf T -\mathbf T^0)\mathbf\Lambda\|_{\mathcal F,\Pi}^2
+\mathfrak c_{\ref{explicit_risk_bound}}\left[k(d +\tau)\frac{\log(n)}{n}
+\frac{1}{n}\log\left(\frac{4}{\alpha}\right)\right]
\end{displaymath}
with probability larger than $1 -\alpha$, where $\mathfrak c_{\ref{explicit_risk_bound}} $ is a constant depending only on $\mathfrak m_0$, $\mathfrak v_\xi$, $\mathfrak c_\xi$, $\mathfrak m_\varepsilon$, $\mathfrak m_{\bf\Lambda}$, $\Phi_{\varepsilon}$ and $\mathfrak{c}_\Pi$.
\end{theorem}
\noindent
Actually, from the proof of the theorem, we know $\mathfrak c_{\ref{explicit_risk_bound}} $ explicitly. Indeed,
\begin{displaymath}
\mathfrak c_{\ref{explicit_risk_bound}} =
72\mathfrak m_0\mathfrak m_{\mathbf\Lambda}\mathfrak c_{\xi}
+ 5\mathfrak c_{\ref{risk_bound},1}
+ 9\mathfrak m_0\mathfrak c_{\ref{risk_bound},2}
\end{displaymath}
where $\mathfrak c_{\ref{risk_bound},1}$ and $c_{\ref{risk_bound},2}$ are constants (explicitly given in Theorem~\ref{risk_bound} in Section~\ref{section_proofs}) depending themselves only on $\mathfrak m_0$, $\mathfrak v_{\xi}$, $ \mathfrak c_{\xi}$, $\mathfrak m_{\varepsilon}$, $\mathfrak m_{\bf\Lambda}$, $\Phi_{\varepsilon}$ and $\mathfrak c_{\Pi}$.
\\
\\
\textbf{Remarks:}
\begin{enumerate}
 \item Note that another classic way to formulate the risk bound in Theorem \ref{explicit_risk_bound} is that for every $s > 0$, with probability larger than $1 - e^{-s}$,
 \begin{displaymath}
 \|\widehat{\mathbf\Theta}_{k,\tau} -\mathbf\Theta^0\|_{\mathcal F,\Pi}^{2}
 \leqslant
 3\min_{\mathbf T\in\mathcal S_{k,\tau}}
 \|(\mathbf T -\mathbf T^0)\mathbf\Lambda\|_{\mathcal F,\Pi}^2
 +\overline{\mathfrak c}_{\ref{explicit_risk_bound}}\left[k(d +\tau)\frac{\log(n)}{n}
 +\frac{s}{n}\right].
 \end{displaymath}
 \item The $\phi$-mixing assumption (Assumption~\ref{assumption_phi_mixing}) is known to be restrictive, we refer the reader to~\cite{doukhan-mixing} where it is compared to other mixing conditions. Some examples are provided in Examples 7, 8 and 9 in~\cite{ALW13}, including stationary AR processes with a noise that has a density with respect to the Lebesgue measure on a compact interval. Interestingly,~\cite{ALW13} also discusses weaker notions of dependence. Under these conditions, we could here apply the inequalities used in~\cite{ALW13}, but it is important to note that this would prevent us from taking $\lambda$ of the order of $n$ in the proof of Proposition~\ref{preliminary_risk_bound}. In other words, this would deteriorate the rates of convergence. A complete study of all the possible dependence conditions on $\varepsilon$ goes beyond the scope of this paper.
\end{enumerate}
\subsection{Lower bound}
In the case where $T =\tau$ and $\mathbf\Lambda =\mathbf I_T$, the model in~\cite{KLT11} is included in our model, and corresponds to the case where the temporally dependent noise $\varepsilon$ is null: $\varepsilon = 0$. This means that the lower bound provided by Theorem 5 in~\cite{KLT11} holds in our setting. That is, when the $\xi_i$'s are $\mathcal N(0,1)$ and the $\mathbf X_i$'s are uniform (so $\mathfrak c_{\Pi} = 1$), there are absolute constants $\mathfrak c_{\inf},\beta > 0$ such that for any $k\leqslant\frac{n}{d\vee T}$,
\begin{displaymath}
\inf_{\widehat{\bf A}}
\sup_{\mathbf\Theta^0\in\mathcal M_{d,k,T}}
\mathbb P\left(\|\widehat{\bf A} -
\mathbf\Theta^0\|_{\mathcal F,\Pi}^{2}\geqslant
\mathfrak c_{\inf}\frac{k(d + T)}{n}\right)\geqslant 1 -\beta.
\end{displaymath}
In other words, the bound in Theorem~\ref{explicit_risk_bound} is tight, maybe up to the $\log(n)$ term (there is also a $\log$ term in the upper bounds of~\cite{KLT11}). We now extend this result to the case $\tau\leqslant T$, in the special case where the deterministic component of the series is $\tau$-periodic: ${\bf\Lambda} = ({\bf I}_{\tau}|\dots|{\bf I}_{\tau})$.
%


%
\begin{theorem}\label{lower_bound}
Assume the $\xi_i$'s are $\mathcal N(0,1)$, the $\mathbf X_i$'s are uniform (so $\mathfrak c_{\Pi} = 1$) and the temporally dependent noise $\varepsilon = 0$. There are absolute constants $\mathfrak c_{\inf},\beta > 0$ such that for any $\tau\in\{1,\dots,T\}$, in the case ${\bf\Lambda} = ({\bf I}_{\tau}|\dots|{\bf I}_{\tau})$, for any $k\leqslant (256\mathfrak m_{0}^{2}n/(d\vee\tau))^{1/3}$,
\begin{displaymath}
\inf_{\widehat{\bf A}}
\sup_{\mathbf\Theta^0\in\mathcal M_{d,k,\tau}}
\mathbb P\left(
\|\widehat{\bf A} -\mathbf\Theta^0{\bf\Lambda}\|_{\mathcal F,\Pi}^{2}
\geqslant\mathfrak c_{\inf}\frac{k(d +\tau)}{n}\right)\geqslant 1 -\beta.
\end{displaymath}
\end{theorem}
\noindent
For $\tau=T$, we recover Theorem 5 in~\cite{KLT11}, but our result also guarantees that the bound in Theorem~\ref{explicit_risk_bound} is tight (up to $\log$ terms) even when $\tau < T$.
%


%
\section{Model selection}
\label{section_model_selection}
The purpose of this section is to provide a selection method of the parameter $k$. First, for the sake of readability, $\mathcal S_{k,\tau}$ and $\widehat{\mathbf T}_{k,\tau}$ are respectively denoted by $\mathcal S_k$ and $\widehat{\mathbf T}_k$ in the sequel. The adaptive estimator studied here is $\widehat{\mathbf\Theta} :=\widehat{\mathbf T}\mathbf\Lambda$, where $\widehat{\mathbf T} :=\widehat{\mathbf T}_{\widehat k}$,
\begin{displaymath}
\widehat k\in\arg\min_{k\in\mathcal K}
\{r_n(\widehat{\mathbf T}_k\mathbf\Lambda) +\textrm{pen}(k)\}
\textrm{ with }
\mathcal K =\{1,\dots,k^*\}\subset\mathbb N^*,
\end{displaymath}
and
\begin{displaymath}
\textrm{pen}(k) :=
16\mathfrak c_{{\rm pen}}
\frac{\log(n)}{n}k(d +\tau)
\textrm{ with }
\mathfrak c_{{\rm pen}} =
2\left(\frac{1}{\mathfrak c_{\ref{preliminary_risk_bound}}}\wedge\lambda^*\right)^{-1}.
\end{displaymath}
Note that the value of the constant $\mathfrak c_{\rm pen}$ could be deduced from the proofs. It would however depend on quantities that are unknown in practice, such as $\mathfrak c_\Pi$ or $\Phi_\varepsilon$. Moreover, the value of $\mathfrak c_{\rm pen}$ provided by the proofs would probably be too large for practical purposes. In practice, we recommend to use the slope heuristics to estimate this constant. The slope heuristic is defined as follows: for each $C > 0$, let us define
\begin{displaymath}
k(C)\in\arg\min_{k\in\mathcal K}
\{r_n(\widehat{\mathbf T}_k\mathbf\Lambda) + C\cdot k\}.
\end{displaymath}
Then, let us define $\widetilde{C}$ as the location of the largest jump of the function
\begin{displaymath}
C\longmapsto r_n(\widehat{\mathbf T}_{k(C)}\mathbf\Lambda)
\end{displaymath}
and choose the rank $\widetilde k = k(2C)$. This popular procedure leads to good practical results in most situations. Its theoretical properties are available only in limited situations (see~\cite{ARLOT19}), though, so we will focus our theoretical result to $\widehat k$.
%


%
\begin{theorem}\label{risk_bound_adaptive_estimator}
Under Assumptions \ref{assumption_phi_mixing}, \ref{assumption_boundedness} and \ref{assumption_probabilities}, if $n\geqslant \max(d,\tau)$, then
\begin{eqnarray*}
 \|\widehat{\mathbf\Theta} -\mathbf\Theta^0\|_{\mathcal F,\Pi}^{2}
 & \leqslant &
 4\min_{k\in\mathcal K}\left\{
 3\min_{\mathbf T\in\mathcal S_k}
 \|(\mathbf T -\mathbf T^0)\mathbf\Lambda\|_{\mathcal F,\Pi}^2 +
 \mathfrak c_{\ref{risk_bound_adaptive_estimator},1}k(d +\tau)\frac{\log(n)}{n}
 \right\}\\
 & &
 \hspace{6cm}
 +\frac{\mathfrak c_{\ref{risk_bound_adaptive_estimator},1}}{n}\log\left(\frac{4k^*}{\alpha}\right)
 +\frac{\mathfrak c_{\ref{risk_bound_adaptive_estimator},2}}{n}
\end{eqnarray*}
with probability larger than $1 -\alpha$, where
\begin{displaymath}
\mathfrak c_{\ref{risk_bound_adaptive_estimator},1} =
4\mathfrak c_{\ref{explicit_risk_bound}} + 16\mathfrak c_{{\rm pen}}
+ 72\mathfrak m_0\mathfrak c_{\xi}
\textrm{ and }
\mathfrak c_{\ref{risk_bound_adaptive_estimator},2} =
9\mathfrak c_{\ref{risk_bound},2}\mathfrak m_0.
\end{displaymath}
\end{theorem}
%


%
\section{Numerical experiments}\label{section_experiments}
This section describes an experimental study of the estimator of the matrix $\mathbf T^0$ introduced at Section \ref{section_setting}. In particular, we compare on simulated periodic data the completion procedure using the periodicity information, to the standard procedure, and we observe a clear improvement. We also illustrate our results on real data from Paris sharing bike system.\\
In the case where no particular temporal structure is used, that is, ${\bf \Lambda}={\bf I}_T$, standard packages such as softImpute~\cite{softImpute} could be used. However, this is not necessarily the case for a general $\Lambda$, thus we implemented a standard alternate least square (ALS) procedure. That is, we iterate ${\bf U} := \arg\min_{U} {\bf r_n}(U {\bf V} {\bf \Lambda}) $ and ${\bf V} := \arg\min_{V} {\bf r_n}({\bf U} V {\bf \Lambda})$ until convergence. Each step is a linear regression and has an explicit solution. Despite its extreme simplicity, this type of alternate optimization is known to lead to very good results in practice~\cite{ALS}, and such a method is actually used by softImpute~\cite{softImpute}. The code of all the experiments can be found on the third author webpage \url{https://amelierosier8.wixsite.com/website}.
%


%
\subsection{Experiments on simulated datas}
The experiments in this subsection are done on datas simulated the following way:
\begin{enumerate}
 \item We generate a matrix $\mathbf T^0 =\mathbf U^0\mathbf V^0$ with $\mathbf U^0\in\mathcal M_{d,k}(\mathbb R)$ and $\mathbf V^0\in\mathcal M_{k,\tau}(\mathbb R)$. Each entries of $\mathbf U^0$ and $\mathbf V^0$ are generated independently by simulating i.i.d. $\mathcal N(0,1)$ random variables.
 \item We multiply $\mathbf T^0$ by a known matrix $\mathbf{\Lambda}\in\mathcal M_{\tau,T}(\mathbb R)$. This matrix depends on the time series structure assumed on $\mathbf M$. Here, we consider the periodic case: $T = p\tau$, $p\in\mathbb N^*$ and $\mathbf\Lambda = (\mathbf I_{\tau}|\dots|\mathbf I_{\tau})$.
 \item The matrix $\mathbf M$ is then obtained by adding a matrix $\varepsilon$ such that $\varepsilon_{1,.},\dots,\varepsilon_{d,.}$ are generated independently by simulating i.i.d. AR(1) processes with compactly supported error in order to meet the $\phi$-\textit{mixing} condition. We multiply $\varepsilon$ by the coefficient $\sigma_{\varepsilon}$ which value will vary according to the experiments. The goal is to evaluate the impact of adding more noise in the estimation.\\
Only 30\% of the entries of $\mathbf M$, taken randomly, are observed. These entries are then corrupted by i.i.d. observation errors $\xi_1,\dots,\xi_n\rightsquigarrow\mathcal N(0,0.01^2)$. To meet Assumption \ref{assumption_boundedness}, we also consider uniform errors $\xi_1,\dots,\xi_n\rightsquigarrow\mathcal U([-a,a])$, where $a =\sqrt 3/100\approx 0.017$ to keep the same variance than previously. The first experiments will show that the estimation remains quite good even if the $\xi_i$'s are not bounded.
\end{enumerate}
\medskip
\noindent
Given the observed entries, our goal is to complete the missing values of the matrix and check if they correspond to the simulated data in two different cases:
\begin{enumerate}
	\item Our first model doesn't take into account the time series structure in the matrix $\mathbf{M}$. Thus, we simply apply our fonction \texttt{als} to the dataframe containing the values of the noisy entries in addition to their position in the matrix $\mathbf M$ (number of the line $j$ and number of the column $t$ with $1\leqslant j \leqslant d$, $1\leqslant t \leqslant T$). The output of the function gives directly an estimator of the matrix $\mathbf \Theta^0$.
	\item Our second model takes into account the time series structure in $\mathbf{M}$ and more precisely the periodicity of the time series datas. In order to have an estimator of the matrix $\mathbf \Theta^0$, some transformation are required on the data: the fonction \texttt{als} is now applied to the dataframe in which all the observed entries at the position $(j,t)$ ($1\leqslant j \leqslant d$, $1\leqslant t \leqslant T$) are now moved to the position $\mathbf(j, t [\rm mod] \tau)$. The output of this function needs to be remultiplied by $\mathbf\Lambda$ to have an estimator of $\mathbf \Theta^0$.
\end{enumerate}

We will evaluate the MSE of the estimator with respect to several parameters and show that there is a gain to take into account the time series structure in the model. As expected, the more $\mathbf\Theta^0$ is perturbed, either with $\varepsilon$ or $\xi_1,\dots,\xi_n$, the more difficult it is to reconstruct the matrix. In the same way, increasing the value of the rank $k$ will lead to a worse estimation. Finally, we study the effect of replacing the uniform error in each AR(1) by a Gaussian one.
\\
\\
The first experiments are done with $d = 1000$, $T=100$ and $\tau=25$ to be in concordance with the experiments on real data (see subsection 5.3). Here are the MSE obtained for both models, 3 values of the rank $k$ and for two kinds of observation errors $\xi_1,\dots,\xi_n$: Gaussian $\mathcal N(0,0.01^2)$ v.s. uniform $\mathcal U([-0.017,0.017])$. The errors in the AR(1) processes generating the rows of $\varepsilon$ remain uniform $\mathcal U([-1,1])$.
\begin{table}[!h]
\begin{center}
\begin{tabular}{|c|c|c|}
\hline
\hline
MSE & \textbf{$\xi_i\rightsquigarrow\mathcal N(0,0.01^2)$} & \textbf{$\xi_i\rightsquigarrow\mathcal U([-0.017,0.017])$}\\
\hline
Model w/o time series struct. & 0.00012 & 0.00014 \\
\hline
Model with time series struct. & 0.00009 &  0.00010 \\
\hline
\hline
\end{tabular}
\medskip
\caption{MSE for both models, $k = 2$.}\label{rank_2}
\end{center}
\end{table}
\begin{table}[!h]
\begin{center}
\begin{tabular}{|c|c|c|}
\hline
\hline
MSE & \textbf{$\xi_i\rightsquigarrow\mathcal N(0,0.01^2)$} & \textbf{$\xi_i\rightsquigarrow\mathcal U([-0.017,0.017])$}\\
\hline
Model w/o time series struct. & 0.00018 & 0.00022 \\
\hline
Model with time series struct. & 0.00012 & 0.00013 \\
\hline
\hline
\end{tabular}
\medskip
\caption{MSE for both models, $k = 5$.}\label{rank_5}
\end{center}
\end{table}
\begin{table}[!htbp]
\begin{center}
\begin{tabular}{|c|c|c|}
\hline
\hline
MSE & \textbf{$\xi_i\rightsquigarrow\mathcal N(0,0.01^2)$} & \textbf{$\xi_i\rightsquigarrow\mathcal U([-0.017,0.017])$}\\
\hline
Model w/o time series struct. & 0.00026 & 0.00045 \\
\hline
Model with time series struct. & 0.00013 & 0.00017 \\
\hline
\hline
\end{tabular}
\medskip
\caption{MSE for both models, $k = 8$.}\label{rank_8}
\end{center}
\end{table}
\newline
Thus, both of the rank $k$ and the nature of the error considered for the $\xi_i$'s seem to play a key role on the reduction of the MSE. Regarding the rank $k$ ($d$, $T$ and $\tau$ being fixed) being fixed), our numerical results are consistent with respect to the theoretical rate of convergence of order $O(k(d+\tau)\log(n)/n)$ obtained at Theorem \ref{explicit_risk_bound} when we consider the time series structure of the data (see Tables \ref{rank_2}, \ref{rank_5} and \ref{rank_8}). Indeed, for both models, the MSE is increasing when the value of the rank $k$ is higher but this increase is always more significant in the model without time series structure, which is also consistent  with the rate of convergence of order $O(k(d + T)\log(n)/n)$ obtained in this case. Note that when we look at one model at a time, for each tested values of $k$, whatever the distribution of the errors $\xi_1,\dots,\xi_n$ (Gaussian or uniform), the MSE remains of same order with a slight improvement when we considered Gaussian errors. This justifies to take $\xi_1,\dots,\xi_n\rightsquigarrow\mathcal N(0,0.01^2)$ in the following experiments.
\\
\\
This study can be summarized in the following experiment which shows the evolution of the MSE with respect to the rank $k$ ($k = 1,\dots,10$) for both models. Once again, we take $d = 1000$, $T = 100$, $\tau=25$ but the $\xi_i$'s remain i.i.d. $\mathcal N(0,0.01^2)$ random variables, and $\varepsilon_{1,.},\dots,\varepsilon_{d,.}$ are i.i.d. AR(1) processes with Gaussian errors.
\begin{figure}[!h]
\begin{center} 
\includegraphics[scale=0.45]{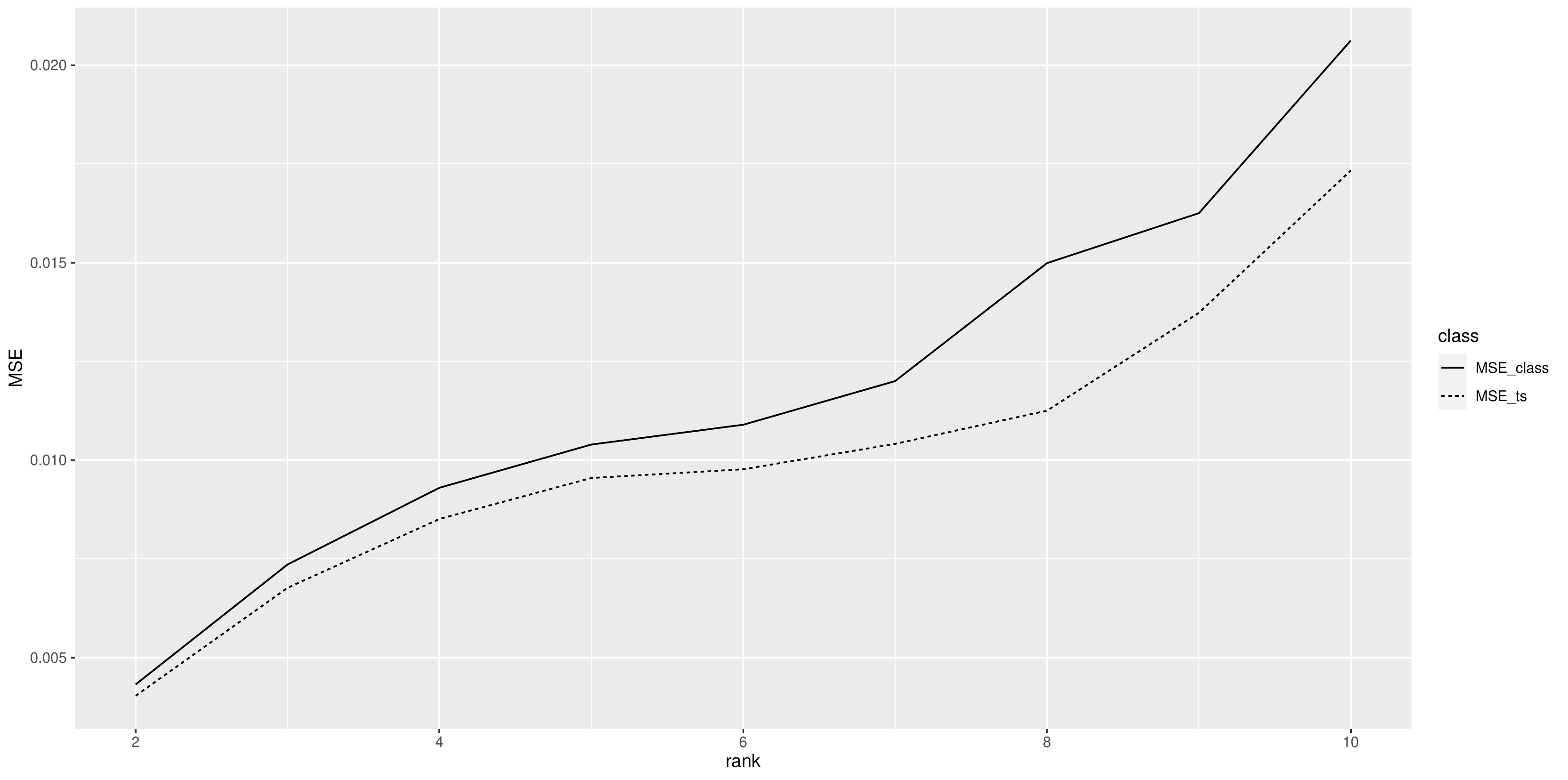}
\caption{Models (time series (dotted line) v.s. classic (solid line)) MSEs with respect to the rank $k$.}\label{figure_rg_MSE}
\end{center}
\end{figure}
\newline
As expected (see Figure \ref{figure_rg_MSE}), the MSE is much better with the model taking into account the time series structure. The MSE in both cases degrades when the value of the rank is increasing, the maximum being reached for $k=10$ with the value 0.0173 for the time series model compared to 0.0206 in the classic case, which still remains very low.
\newline
\newline
As we said, the estimation seems to be more precise with Gaussian errors in $\varepsilon$, and the more $\mathbf\Theta^0$ is perturbed via $\varepsilon$ or $\xi_1,\dots,\xi_n$, the more the completion process is complicated and the MSE degrades. So, we now evaluate the consequence on the MSE of changing the value of $\sigma_{\varepsilon}$. For both models (with or without taking into account the time series structure), the following figure shows the evolution of the MSE with respect to $\sigma_{\varepsilon}$ when the errors in $\varepsilon$ are $\mathcal N(0,1/3)$ random variables and all the other parameters remain the same than previously, we are still considering 30\% of observed entries. 
\begin{figure}[!h]
\begin{center} 
\includegraphics[scale=0.45]{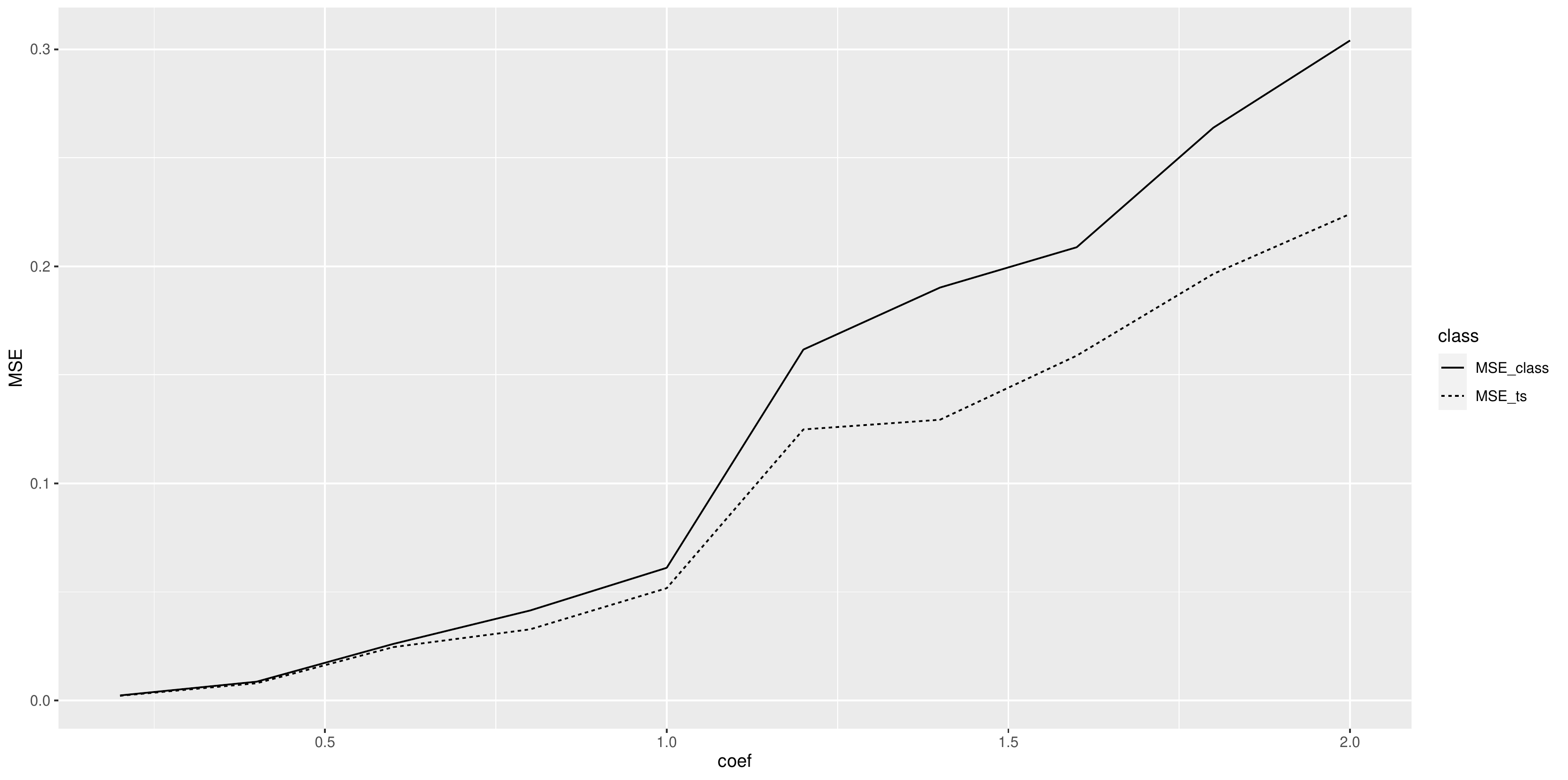}
\caption{Models (time series (dotted line) v.s. classic (solid line)) MSEs with respect to $\sigma_{\varepsilon}$, Gaussian errors.}\label{figure_coef_MSE_1}
\end{center}
\end{figure}
\newline
Once again, as expected (see Figure \ref{figure_coef_MSE_1}), the MSE with time series model is smaller than the one with the classic model for each values of $\sigma_{\varepsilon}$. The fact that the MSE increases with respect to $\sigma_{\varepsilon}$ with both models illustrates that \textit{more noise} always complicates the completion process. In our experiments, the values of $\sigma_{\varepsilon}$ range from 0.02 to 2. We can notice that, the more we add noise with $\sigma_{\varepsilon}$, the more significant the gap between the MSE of both models is. With $\sigma_\varepsilon$ equal to 2, the MSE reaches the value 0.2241 for the time series model and 0.3040 for the classic one. Our method has increasing difficulty in reconstructing the matrix when we add too much noise to the model. See also Table \ref{min_max_normal}.
\begin{table}[!htbp]
\begin{center}
\begin{tabular}{|c|c|c|}
\hline
\hline
 & {\bf Min. MSE} & {\bf Max. MSE}\\
\hline
Model w/o time series struct. & 0.0023 & 0.3040\\
\hline
Model with time series struct. & 0.0021 & 0.2241\\
\hline
\hline
\end{tabular}
\medskip
\caption{Min. and max. values reached by the MSE with Gaussian errors in $\varepsilon$.\label{min_max_normal}}
\end{center}
\end{table}
\newline
Let us do the same experiment but with uniform $\mathcal U([-1,1])$ errors in the AR(1) processes generating the rows of $\varepsilon$.
\medskip
\begin{figure}[!h]
\begin{center} 
\includegraphics[scale=0.45]{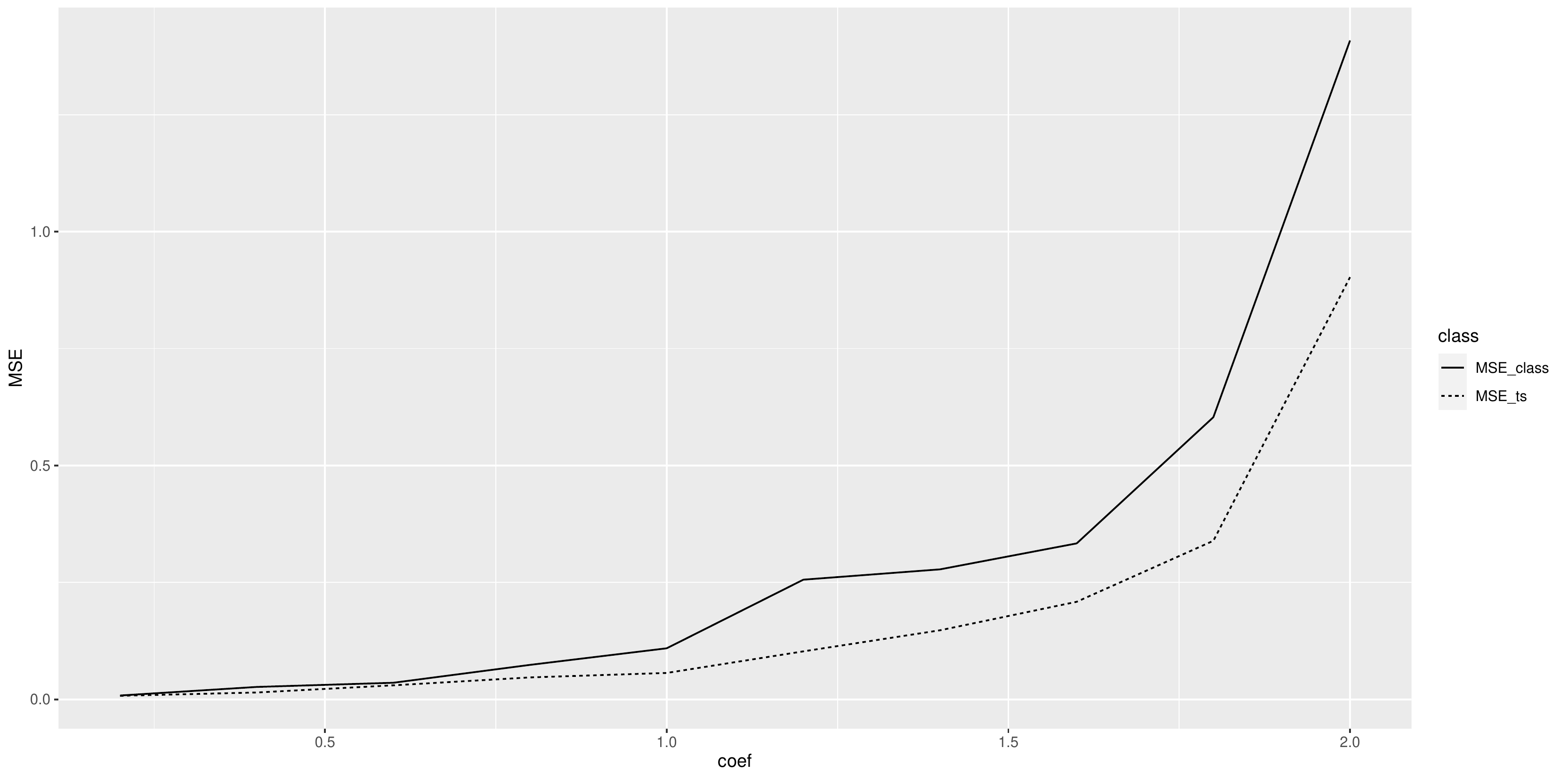}
\caption{Models (time series (dotted line) v.s. classic (solid line)) MSEs with respect to $\sigma_{\varepsilon}$, uniform errors.}\label{figure_coef_MSE_2}
\end{center}
\end{figure}
\newline
The curves shape on Figure \ref{figure_coef_MSE_2} is pretty much the same as in the previous graph: the MSE for the model taking into account the time series structure is still smaller than for the classic model and this difference between the two models is even greater when we increase the value of $\sigma_{\varepsilon}$. However, this time, the MSE for both models reaches higher values, leading to a huge misestimation when $\sigma_{\varepsilon}=2$ (see Table \ref{min_max_uniform}).
\begin{table}[!htbp]
\begin{center}
\begin{tabular}{|c|c|c|}
\hline
\hline
 & {\bf Min. MSE} & {\bf Max. MSE}\\
\hline
Model w/o time series struct. & 0.0082 & 1.4088\\
\hline
Model with time series struct. & 0.0076 & 0.9027\\
\hline
\hline
\end{tabular}
\medskip
\caption{Min. and max. values reached by the MSE with uniform errors in $\varepsilon$.}\label{min_max_uniform}
\end{center}
\end{table}
\newline
\newline
Finally, as mentioned, the previous numerical experiments were done by assuming that $k$ is known, which is mostly uncommon in practice. So, our purpose in the last part of this section is to implement the model selection method introduced at Section \ref{section_model_selection}. Let us recall the criterion to minimize:
\begin{displaymath}
\left\{
\begin{array}{rcl}
 \textrm{crit}(k) & = &
 r_n(\widehat{\mathbf T}_k\mathbf\Lambda) +\textrm{pen}(k)\\
 \textrm{pen}(k) & = &
 \mathfrak c_{\rm cal}
 k(d +\tau)\log(n)/n
\end{array}\right.
\textrm{$;$ }
k \in\{1,\dots,20\}.
\end{displaymath}
In the sequel, $\xi_1,\dots,\xi_n\rightsquigarrow\mathcal N(0,0.5)$, $\varepsilon_{1,.},\dots,\varepsilon_{d,.}$ are i.i.d. AR(1) processes with $\mathcal N(0,1/3)$ errors, and $\sigma_{\varepsilon} = 0.2$. Percentage of observed entries is still 30\%. The penalty term in ${\rm crit}(.)$ depends on the constant $\mathfrak c_{\rm cal} > 0$ which is calibrated here by using the slope heuristic presented at Section \ref{section_model_selection}. \\
On 20 independent experiments, Table \ref{table_k} gives the mean MSE obtained for the estimator computed with the true rank $k=5$ and the associated adaptative estimator computed with $\widehat k$ selected by minimizing the criterion studied in Section \ref{section_model_selection}.
\begin{table}[!htbp]
\begin{center}
\begin{tabular}{|c|c|}
\hline
\hline
{\bf Mean MSE for $\widehat{\mathbf T}_k\mathbf\Lambda$} &  0.10712 \\
\hline
\textbf{Mean MSE for $\widehat{\mathbf T}_{\widehat k}\mathbf\Lambda$} & 0.17601 \\
\hline
\hline
\end{tabular}
\medskip
\caption{Mean MSE over 20 simulations for $\widehat{\mathbf T}_k\mathbf\Lambda$ and $\widehat{\mathbf T}_{\widehat k}\mathbf\Lambda$.}\label{table_k}
\end{center}
\end{table}
Table \ref{frequence} gives the frequence of the different values of $k$ selected. Our method select the true $k$ 8 times over 20.
\begin{table}[!htbp]
\begin{center}
\begin{tabular}{|c|c|c|c|c|c|c|}
\hline
\hline
\textbf{k selected} & 4 & 5 & 6 & 7 & 8 & 9 \\
\hline
\textbf{Frequence} & 0.05 & 0.4 & 0.1 & 0.15 & 0.2 & 0.1 \\
\hline
\hline
\end{tabular}
\medskip
\caption{Frequence of k-values selected}\label{frequence}
\end{center}
\end{table}

%
%

%
\subsection{Experiments on real datas}
Modern transportation data are often high-dimensional and have strong patterns including periodicity. For this reason, matrix factorization methods are very popular in this field~\cite{C19,T+18}.
The data used in this section comes from the \texttt{funFEM} package (the real time data are available at \url{https://developer.jcdecaux.com/}). We used the \textit{Velib} data set which contains data from the bike sharing system of Paris. These data provide the occupancy (number of available bikes/number of bike docks) of 1189 bike stations over one week. The data were collected every hour during the following period: Sunday 1st Sept. - Sunday 7th Sept., 2014. We removed the time points collected during the week-end (50 time points in total) insofar as the week-end occupancy of the bike stations differs from the week. Loading profiles of 6 different stations (week-end excluded) are represented on Figure \ref{bike_stations_examples}.

\medskip
\begin{figure}[!h]
\begin{center} 
\includegraphics[scale=1]{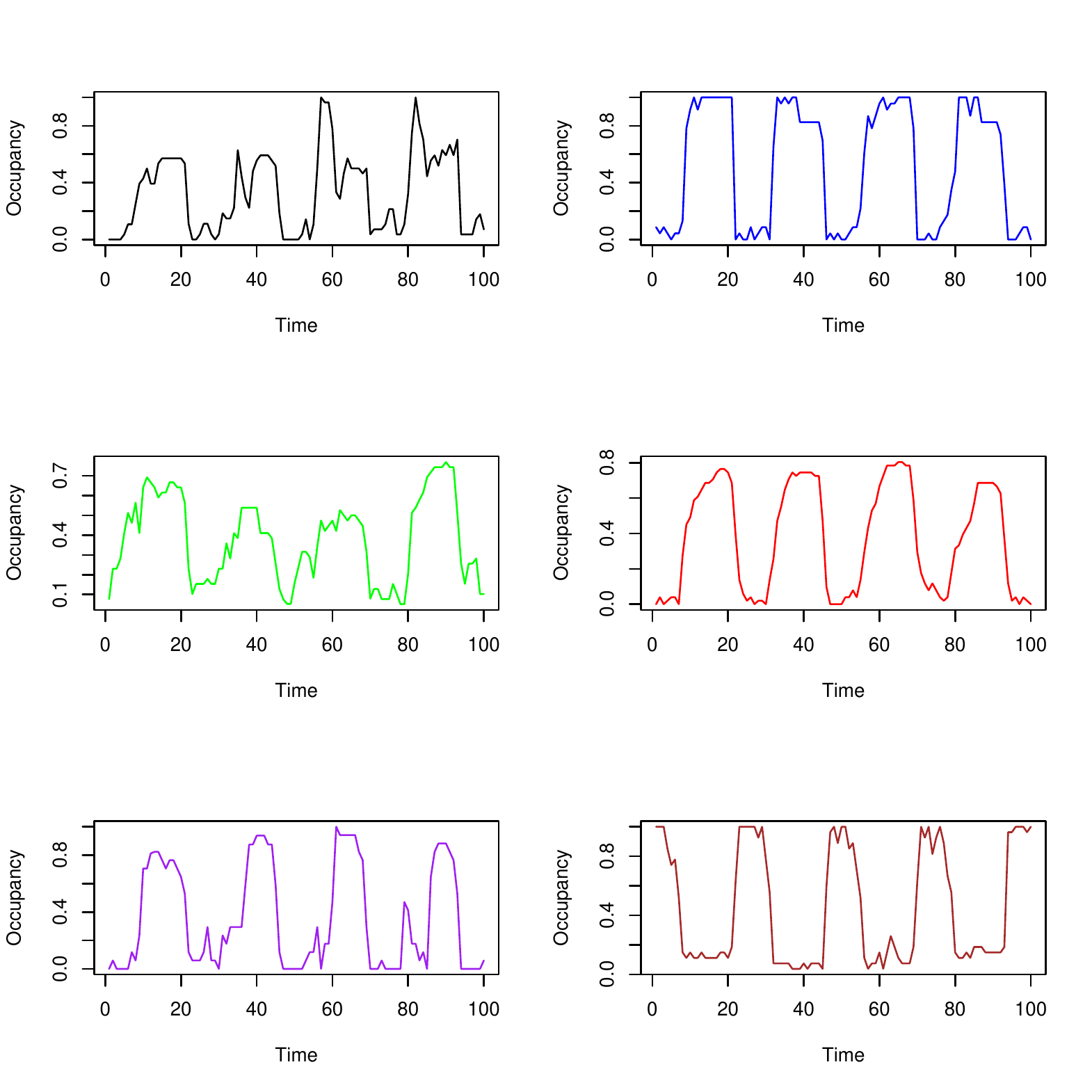}
\caption{Occupancy of six \textit{Velib} stations over one week (week-end excluded).}\label{bike_stations_examples}
\end{center}
\end{figure}

We clearly notice the daily periodic behaviour of our time series. Thus, the experiments of this section are done with the real time data in the matrix $\mathbf M$ of dimensions $d=1189$, $\tau=25$ (which corresponds to one day) and $T=125$ (four days, from Monday to Thursday). Once again, we evaluate the MSE of the estimator with and without taking into account the time series structure, that is the periodicity in this case. Different percentages of the entries observed are tested. As for the simulated data, for the model without considering the temporal structure of our series, we apply directly our function \texttt{als} on the dataframe containing the observed entries with their position in the matrix, without any additional transformation on the data. The output gives directly an estimator of $\mathbf M$. As regards the model considering the periodic behaviour of the \textit{Velib} time series in $\mathbf M$, the ALS optimization procedure is applied on the dataframe which has received the same transformation  than the one explained at point (2) in  the previous section. Once again, the output needs to be multiplied by $\mathbf\Lambda$ to have an estimator of $\mathbf M$ at the end. The MSEs obtained for both models are gathered in Table \ref{mean_MSE_real_data}. We study how the MSEs vary according to the percentage of observed entries.
\newline

\begin{table}[!htbp]
\begin{center}
\begin{tabular}{|c|c|c|c|}
\hline
\hline
 & 15\% & 30\% \\
\hline
\textbf{Model w/o time series struct.} & 0.0609 & 0.0315 \\
\hline
\textbf{Model with time series struct.} & 0.0436 & 0.0381 \\
\hline
\hline
\end{tabular}
\medskip
\caption{MSE according to the number of observed entries (\%). }\label{mean_MSE_real_data}
\end{center}
\end{table}
Of course, the real data is not {\it exactly} periodic (as can bee seen in some of the series in Figure~\ref{bike_stations_examples}. This means that the bias term of in Theorem~\ref{explicit_risk_bound}, is larger for the method imposing periodicity than for the standard method: $\min_{{\bf T}\in \mathcal{S}_{k,\tau}} \|({\bf T}-{\bf T}^0){\bf \Lambda}\|_{\mathcal{F},\Pi} \geq \min_{{\bf T}\in \mathcal{S}_{k,T}} \|{\bf T}-{\bf T}^0\|_{\mathcal{F},\Pi}$. On the other hand, the variance term of the method using periodicity is much smaller: $k(d+\tau)/n \leq k(d+T)/n$. Thus, it is expected that when the sample size $n$ is small, using periodicity can improve on the standard method, but that this is not the case for larger values of $n$. This is perfectly illustrated by our experiments: Table~\ref{mean_MSE_real_data} show that when we observe 15\% of the original data, exploiting periodicity improves on the reconstruction of the data by the standard method by more than 25\%. On the other hand, when we the sample size doubles, the standard method already performs slightly better.
%


%
\section{Proofs}
\label{section_proofs}
This section is organized as follows. We first state an exponential inequality that will serve as a basis for all the proofs. From this inequality, we prove Theorem~\ref{risk_bound}, a prototype of Theorem~\ref{explicit_risk_bound} that holds when the set $\mathcal S_{k,\tau}$ is finite or infinite but compact by using $\epsilon$-nets ($\epsilon > 0$). In the proof of Theorem~\ref{explicit_risk_bound}, we provide an explicit risk-bound by using the $\epsilon$-net $\mathcal S_{k,\tau}^{\epsilon}$ of $\mathcal S_{k,\tau}$ constructed in Cand\`es and Plan \cite{CP11}, Lemma 3.1.
%


%
\subsection{Exponential inequality}
This sections deals with the proof of the following exponential inequality, the cornerstone of the paper, which is derived from the usual Bernstein inequality and its extension to $\phi$-mixing processes due to Samson \cite{SAMSON00}.
\begin{proposition}\label{preliminary_risk_bound}
Let $\mathbf T \in S_{k,\tau}$.
Under Assumptions \ref{assumption_phi_mixing}, \ref{assumption_boundedness} and \ref{assumption_probabilities},
\begin{equation}\label{preliminary_risk_bound_1}
\mathbb E\left[\exp\left(
\frac{\lambda}{4}\left(\left(1 +\mathfrak c_{\ref{preliminary_risk_bound}}\frac{\lambda}{n}\right)
(R(\mathbf T^0\mathbf\Lambda) - R(\mathbf T\mathbf\Lambda))
+ r_n(\mathbf T\mathbf\Lambda) - r_n(\mathbf T^0\mathbf\Lambda)\right)\right)\right]\leqslant 1
\end{equation}
and
\begin{equation}\label{preliminary_risk_bound_2}
\mathbb E\left[\exp\left(
\frac{\lambda}{4}\left(\left(1 -\mathfrak c_{\ref{preliminary_risk_bound}}\frac{\lambda}{n}\right)
(R(\mathbf T\mathbf\Lambda) - R(\mathbf T^0\mathbf\Lambda))
+ r_n(\mathbf T^0\mathbf\Lambda) - r_n(\mathbf T\mathbf\Lambda)\right)\right)\right]\leqslant 1
\end{equation}
for every $\mathbf T\in\mathcal S_{k,\tau}$ and $\lambda\in (0,n\lambda^*)$, where
\begin{displaymath}
R(\mathbf A) :=\mathbb E(|Y_1 -\langle\mathbf X_1,\mathbf A\rangle_{\mathcal F}|^2)
\textrm{ $;$ }
\forall\mathbf A\in\mathcal M_{d,T}(\mathbb R),
\end{displaymath}
$\mathfrak c_{\ref{preliminary_risk_bound}} = 4\max\{4\mathfrak m_{0}^{2},4\mathfrak v_{\xi},4\mathfrak m_{\varepsilon}^{2},2\mathfrak m_{\varepsilon}^{2}\Phi_{\varepsilon}^{2}\mathfrak c_{\Pi}\}$ and $\lambda^* = (16\mathfrak m_0\max\{\mathfrak m_0,\mathfrak m_{\varepsilon},\mathfrak c_{\xi}\})^{-1}$.
\end{proposition}
%


%
\begin{proof}[Proof of Proposition \ref{preliminary_risk_bound}]
The proof relies on Bernstein's inequality as stated in \cite{NLM13}, that we remind in the following lemma.
%


%
\begin{lemma}\label{Bernstein_independent_rv}
Let $T_1,\dots,T_n$ be some independent and real-valued random variables. Assume that there are $v > 0$ and $c > 0$ such that
\begin{displaymath}
\sum_{i = 1}^{n}
\mathbb E(T_{i}^{2})
\leqslant v
\end{displaymath}
and, for any $q\geqslant 3$,
\begin{displaymath}
\sum_{i = 1}^{n}
\mathbb E(T_{i}^{q})
\leqslant
\frac{vc^{q - 2}q!}{2}.
\end{displaymath}
Then, for every $\lambda\in(0,1/c)$,
\begin{displaymath}
\mathbb E
\left[\exp\left[\lambda\sum_{i = 1}^{n}
(T_i - \mathbb E(T_i))\right]\right]
\leqslant
\exp\left(\frac{v\lambda^2}{2(1 - c\lambda)}\right).
\end{displaymath}
\end{lemma}
\noindent
We will also use a variant of this inequality for time series due to Samson, stated in the proof of Theorem 3 in \cite{SAMSON00}.
%


%
\begin{lemma}\label{Bernstein_dependent_rv}
Consider $m\in\mathbb N^*$, $M > 0$, a stationary sequence of $\mathbb R^m$-valued random variables $Z = (Z_t)_{t\in\mathbb Z}$, and
\begin{displaymath}
\Phi_Z :=
1 +\sum_{t = 1}^{T}\phi_Z(t)^{1/2},
\end{displaymath}
where $\phi_Z(t)$, $t\in\mathbb Z$, are the $\phi$-mixing coefficients of $Z$. For every smooth and convex function $f : [0,M]^T\rightarrow\mathbb R$ such that $\|\nabla f\|\leqslant L$ a.e, for any $\lambda>0$,
\begin{displaymath}
\mathbb E(\exp(\lambda(f(Z_1,\dots,Z_T) - \mathbb E[f(Z_1,\dots,Z_T)])))
\leqslant
\exp\left(\frac{
\lambda^2L^2\Phi_{Z}^{2}M^2}{2}\right).
\end{displaymath}
\end{lemma}
\noindent
Let $\mathbf T\in\mathcal S_{k,\tau}$ be arbitrarily chosen. Consider the deterministic map $\mathbf X :\mathcal E\rightarrow\mathcal M_{d,T}(\mathbb R)$ such that
\begin{displaymath}
\mathbf X_i =\mathbf X(\chi_i)
\textrm{ $;$ }
\forall i\in\{1,\dots,n\},
\end{displaymath}
$\Xi_i := (\overline\xi_i,\chi_i)$ for any $i\in\{1,\dots,n\}$, and $h :\mathbb R\times\mathcal E\rightarrow\mathbb R$ the map defined by
\begin{displaymath}
h(x,y) :=
\frac{1}{n}(2x\langle\mathbf X(y),(\mathbf T^0 -\mathbf T)\mathbf\Lambda\rangle_{\mathcal F}
+\langle\mathbf X(y),(\mathbf T^0 -\mathbf T)\mathbf\Lambda\rangle_{\mathcal F}^{2})
\textrm{ $;$ }
\forall (x,y)\in\mathbb R\times\mathcal E.
\end{displaymath}
Note that
\begin{eqnarray*}
 h(\Xi_i) & = &
 \frac{1}{n}(2\overline\xi_i\langle\mathbf X_i,(\mathbf T^0 -\mathbf T)\mathbf\Lambda\rangle_{\mathcal F}
 +\langle\mathbf X_i,(\mathbf T^0 -\mathbf T)\mathbf\Lambda\rangle_{\mathcal F}^{2})\\
 & = &
 \frac{1}{n}((\overline\xi_i +\langle\mathbf X_i,(\mathbf T^0 -\mathbf T)\mathbf\Lambda\rangle_{\mathcal F})^2 -\overline\xi_{i}^{2})\\
 & = &
 \frac{1}{n}(
 (Y_i -\langle\mathbf X_i,\mathbf T\mathbf\Lambda\rangle_{\mathcal F})^2 -
 (Y_i -\langle\mathbf X_i,\mathbf T^0\mathbf\Lambda\rangle_{\mathcal F})^2)
\end{eqnarray*}
and
\begin{displaymath}
\sum_{i = 1}^{n}
(h(\Xi_i) -\mathbb E(h(\Xi_i))) =
r_n(\mathbf T\mathbf\Lambda) -
r_n(\mathbf T^0\mathbf\Lambda) +
R(\mathbf T^0\mathbf\Lambda) -
R(\mathbf T\mathbf\Lambda).
\end{displaymath}
Now, replacing $\overline{\xi}_i$ by its expression in terms of $\mathbf X_i$, $\xi_i$ and $\varepsilon$,
\begin{eqnarray*}
 \sum_{i = 1}^{n}
 (h(\Xi_i) -\mathbb E(h(\Xi_i)))
 & = &
 \sum_{i = 1}^{n}
 \left(\frac{2}{n}\xi_i
 \langle\mathbf X_i,(\mathbf T^0 -\mathbf T)\mathbf\Lambda\rangle_{\mathcal F}\right)\\
 & &
 +\sum_{i = 1}^{n}
 \left(\frac{2}{n}\langle\mathbf X_i,\varepsilon\rangle_{\mathcal F}\langle\mathbf X_i,(\mathbf T^0 -\mathbf T)\mathbf\Lambda\rangle_{\mathcal F}\right)\\
 & &
 +\sum_{i =1}^{n}
 \left(\frac{1}{n}\langle\mathbf X_i,(\mathbf T^0 -\mathbf T)\mathbf\Lambda\rangle_{\mathcal F}^{2} - \mathbb E(h(\Xi_i))\right)\\
 & =: &
 \sum_{i = 1}^{n}A_i +
 \sum_{i = 1}^{n}B_i +
 \sum_{i = 1}^{n}(C_i -\mathbb E(h(\Xi_i))).
\end{eqnarray*}
In order to conclude, by using Lemmas \ref{Bernstein_independent_rv} and \ref{Bernstein_dependent_rv}, let us provide suitable bounds for the exponentiel moments of each terms of the previous decomposition:
\begin{itemize}
 \item\textbf{Bounds for the $A_i$'s and the $C_i$'s.} First, note that since $\mathbf X_1$, $\xi_1$ and $\varepsilon$ are independent,
 \begin{eqnarray}
  R(\mathbf T\mathbf\Lambda) - R(\mathbf T^0\mathbf\Lambda) & = &
  \mathbb E((Y_1 -\langle\mathbf X_1,\mathbf T\mathbf\Lambda\rangle_{\mathcal F})^2 -
  (Y_1 -\langle\mathbf X_1,\mathbf T^0\mathbf\Lambda\rangle_{\mathcal F})^2)
  \nonumber\\
  & = &
  2\mathbb E(\overline\xi_1\langle\mathbf X_1,(\mathbf T^0 -\mathbf T)\mathbf\Lambda\rangle_{\mathcal F})
  +\mathbb E(\langle\mathbf X_1,(\mathbf T^0 -\mathbf T)\mathbf\Lambda\rangle_{\mathcal F}^{2})
  \nonumber\\
  & = &
  2\langle\mathbb E(\langle\mathbf X_1,(\mathbf T^0 -\mathbf T)\mathbf\Lambda\rangle_{\mathcal F}\mathbf X_1),\mathbb E(\varepsilon)\rangle_{\mathcal F}\nonumber\\
  & &
  + 2\mathbb E(\xi_1)\mathbb E(\langle\mathbf X_1,(\mathbf T^0 -\mathbf T)\mathbf\Lambda\rangle_{\mathcal F})
  + \|(\mathbf T^0 -\mathbf T)\mathbf\Lambda\|_{\mathcal F,\Pi}^{2}
  \nonumber\\
  \label{preliminary_risk_bound_3}
  & = &
  \|(\mathbf T^0 -\mathbf T)\mathbf\Lambda\|_{\mathcal F,\Pi}^{2}.
 \end{eqnarray}
 On the one hand,
 \begin{displaymath}
 \mathbb E(A_{i}^{2})
 \leqslant\frac{4}{n^2}\mathbb E(\xi_{i}^{2})
 \mathbb E(\langle\mathbf X_i,
 (\mathbf T^0 -\mathbf T)\mathbf\Lambda\rangle_{\mathcal F}^{2})
 \leqslant \frac{4}{n^2}
 \mathfrak v_{\xi}(R(\mathbf T^0\mathbf\Lambda)
 - R(\mathbf T\mathbf\Lambda))
 \end{displaymath}
 thanks to Equality (\ref{preliminary_risk_bound_3}). Moreover,
 \begin{eqnarray*}
  \mathbb E(|A_i|^q)
  & \leqslant &
  \frac{2^q}{n^q}\mathbb E(|\xi_i|^q)
  \mathbb E(\langle\mathbf X_i,(\mathbf T^0 -\mathbf T)
  \mathbf\Lambda\rangle_{\mathcal F}^{q})\\
  & \leqslant &
  \left(\frac{4\mathfrak c_{\xi}\mathfrak m_0}{n}\right)^{q - 2}
  \frac{q!}{2}\cdot
  \frac{4\mathfrak v_{\xi}}{n^2}
  (R(\mathbf T^0\mathbf\Lambda)
  - R(\mathbf T\mathbf\Lambda)).
 \end{eqnarray*}
 So, we can use Lemma \ref{Bernstein_independent_rv} with
 \begin{displaymath}
 v =\frac{4}{n}\mathfrak v_{\xi}(
 R(\mathbf{T}^{0}\mathbf\Lambda)
 - R(\mathbf T\mathbf\Lambda))
 \textrm{ and }
 c =\frac{4\mathfrak c_{\xi}\mathfrak m_0}{n}
 \end{displaymath}
 to obtain:
 \begin{displaymath}
 \mathbb E\left[\exp\left(\lambda \sum_{i = 1}^{n} A_i\right)\right]
 \leqslant
 \exp\left[\frac{2\mathfrak v_{\xi}
 (R(\mathbf T^0\mathbf\Lambda) - R(\mathbf T\mathbf\Lambda))
 \lambda^2}{n - 4\mathfrak c_{\xi}\mathfrak m_0\lambda}\right]
 \end{displaymath}
 for any $\lambda\in (0,n/(4\mathfrak c_{\xi}\mathfrak m_0))$. On the other hand, $|C_i|\leqslant 4\mathfrak m_{0}^{2}/n$ and
 \begin{displaymath}
 \mathbb E(C_{i}^{2}) =
 \frac{1}{n^2}
 \mathbb E(\langle\mathbf X_i,(\mathbf T^0 -\mathbf T)\mathbf\Lambda\rangle_{\mathcal F}^{4})
 \leqslant
 \frac{4\mathfrak m_{0}^{2}}{n^2}
 \|(\mathbf T^0 -\mathbf T)\mathbf\Lambda\|_{\mathcal F,\Pi}^{2} =
 \frac{4}{n^2}
 \mathfrak m_{0}^{2}
 (R(\mathbf T^0\mathbf\Lambda) - R(\mathbf T\mathbf\Lambda))
 \end{displaymath}
 thanks to Equality (\ref{preliminary_risk_bound_3}). So, we can use Lemma \ref{Bernstein_independent_rv} with
 \begin{displaymath}
 v =\frac{4}{n}\mathfrak m_{0}^{2}(
 R(\mathbf{T}^{0}\mathbf\Lambda)
 - R(\mathbf T\mathbf\Lambda))
 \textrm{ and }
 c =\frac{4\mathfrak m_{0}^{2}}{n}
 \end{displaymath}
 to obtain:
 \begin{displaymath}
 \mathbb E\left[\exp\left(\lambda\sum_{i = 1}^{n}
 (C_i -\mathbb E(h(\Xi_i)))\right)\right]
 \leqslant
 \exp\left[\frac{2\mathfrak m_{0}^{2}
 (R(\mathbf T^0\mathbf\Lambda) - R(\mathbf T\mathbf\Lambda))
 \lambda^2}{n - 4\mathfrak m_{0}^{2}\lambda}\right]
 \end{displaymath}
 for any $\lambda\in (0,n/(4\mathfrak m_{0}^{2}))$.
 \item\textbf{Bounds for the $B_i$'s.} First, write
 \begin{displaymath}
 \sum_{i = 1}^{n}B_i
 =\sum_{i = 1}^{n}
 (B_i - \mathbb E(B_i|\varepsilon))
 +\sum_{i = 1}^{n}
 \mathbb E(B_i|\varepsilon)
 =:\sum_{i = 1}^{n}D_i + \sum_{i = 1}^{n}E_i,
 \end{displaymath}
 and note that
 \begin{eqnarray}
  \label{preliminary_risk_bound_4}
  \mathbb E(B_i|\varepsilon) & = &
  \frac{2}{n}
  \mathbb E(\langle\mathbf X_i,\varepsilon\rangle_{\mathcal F}
  \langle\mathbf X_i,(\mathbf T^0 -\mathbf T)\mathbf\Lambda\rangle_{\mathcal F}|\varepsilon)\\
  & = &
  \frac{2}{n}
  \sum_{j,t}
  \mathbb E(\mathbf 1_{\chi_i = (j,t)}
  [(\mathbf T^0 -\mathbf T)\mathbf\Lambda]_{\chi_i})\varepsilon_{j,t}
  =\frac{2}{n}
  \sum_{j,t}
  p_{j,t}[(\mathbf T^0 -\mathbf T)\mathbf\Lambda]_{j,t}\varepsilon_{j,t}
  \nonumber
 \end{eqnarray}
 and
 \begin{equation}\label{preliminary_risk_bound_5}
 \|(\mathbf T^0 -\mathbf T)\mathbf\Lambda\|_{\mathcal F,\Pi}^{2} =
 \mathbb E(\langle\mathbf X_i,(\mathbf T^0 -\mathbf T)\mathbf\Lambda\rangle_{\mathcal F}^{2})
 =\mathbb E([(\mathbf T^0 -\mathbf T)\mathbf\Lambda]_{\chi_i}^{2})
 =\sum_{j,t}p_{j,t}
 [(\mathbf T^0 -\mathbf T)\mathbf\Lambda]_{j,t}^{2},
 \end{equation}
 where
 \begin{displaymath}
 p_{j,t} :=\mathbb P(\chi_1 = (j,t)) =\Pi(\{e_{\mathbb R^d}(j)e_{\mathbb R^T}(t)^*\})
 \end{displaymath}
 for every $(j,t)\in\mathcal E$. On the one hand, given $\varepsilon$, the $D_i$'s are i.i.d, $|D_i|\leqslant 8\mathfrak m_{\varepsilon} \mathfrak m_0/n$ and
 \begin{align*}
 \mathbb E(B_i^2|\varepsilon)
 & =\frac{4}{n^2}\mathbb E(\langle\mathbf X_i,\varepsilon\rangle_{\mathcal F}^{2}
 \langle\mathbf X_i,(\mathbf T^0 -\mathbf T)\mathbf\Lambda\rangle_{\mathcal F}^{2}|\varepsilon)\\
 & \leqslant\frac{4}{n^2}\mathfrak m_{\varepsilon}^{2}
 \mathbb E(\langle\mathbf X_i,(\mathbf T^0 -\mathbf T)\mathbf\Lambda\rangle_{\mathcal F}^{2}|\varepsilon)
 =\frac{4}{n^2}\mathfrak m_{\varepsilon}^{2}\mathbb E(\langle\mathbf X_i,(\mathbf T^0 -\mathbf T)\mathbf\Lambda\rangle_{\mathcal F}^{2})
 = \frac{4}{n^2}\mathfrak m_{\varepsilon}^2(R(\mathbf T^0 \mathbf\Lambda) - R(\mathbf T \mathbf\Lambda))
 \end{align*}
 thanks to Equality (\ref{preliminary_risk_bound_3}). So, {\it conditionnally on $\varepsilon$}, we can apply Lemma \ref{Bernstein_independent_rv} with
 \begin{displaymath}
 v =\frac{4}{n}\mathfrak m_{\varepsilon}^{2}(
 R(\mathbf{T}^{0}\mathbf\Lambda)
 - R(\mathbf T\mathbf\Lambda))
 \textrm{ and }
 c =\frac{8\mathfrak m_{\varepsilon}\mathfrak m_0}{n}
 \end{displaymath}
 to obtain:
 \begin{displaymath}
 \mathbb E\left[\left.\exp\left(\lambda\sum_{i = 1}^{n}D_i\right)\right|\varepsilon\right]
 \leqslant
 \exp\left[\frac{2\mathfrak m_{\varepsilon}^{2}
 (R(\mathbf T^0\mathbf\Lambda) - R(\mathbf T\mathbf\Lambda))
 \lambda^2}{n - 8\mathfrak m_{\varepsilon}\mathfrak m_0\lambda}\right]
 \end{displaymath}
 for any $\lambda\in (0,n/(8\mathfrak m_{\varepsilon}\mathfrak m_0))$. Taking the expectation of both sides gives:
 \begin{displaymath}
 \mathbb E\left[\exp\left(\lambda\sum_{i = 1}^{n}D_i\right)\right]
 \leqslant
 \exp\left[\frac{2\mathfrak m_{\varepsilon}^{2}
 (R(\mathbf T^0\mathbf\Lambda) - R(\mathbf T\mathbf\Lambda))
 \lambda^2}{n - 8\mathfrak m_{\varepsilon}\mathfrak m_0\lambda}\right].
 \end{displaymath}
 On the other hand, let us focus on the $E_i$'s. Thanks to Equality (\ref{preliminary_risk_bound_4}) and since the rows of $\varepsilon$ are independent,
 \begin{eqnarray*}
  \mathbb E\left[\exp\left(\lambda
  \sum_{i = 1}^{n}E_i\right)\right] & = &
  \mathbb E\left[\exp\left[2\lambda
  \sum_{j,t}p_{j,t}[(\mathbf T^0 -\mathbf T)\mathbf\Lambda]_{j,t}\varepsilon_{j,t}\right]\right]\\
  & = &
  \prod_{j = 1}^{d}
  \mathbb E\left[\exp\left(
  2\lambda\sum_{t = 1}^{T}p_{j,t}[(\mathbf T^0 -\mathbf T)\mathbf\Lambda]_{j,t}\varepsilon_{j,t}\right)\right].
 \end{eqnarray*}
 Now, for any $j\in\{1,\dots,d\}$, let us apply Lemma \ref{Bernstein_dependent_rv} to $(\varepsilon_{j,1},\dots,\varepsilon_{j,T})$, which is a sample of a $\phi$-mixing sequence, and to the function $f_j : [0,\mathfrak m_{\varepsilon}]^T\rightarrow\mathbb R$ defined by
 \begin{displaymath}
 f_j(u_1,\dots,u_T) :=
 2\sum_{t = 1}^{T}p_{j,t}[(\mathbf T^0 -\mathbf T)\mathbf\Lambda]_{j,t}u_t
 \textrm{ $;$ }
 \forall u\in [0,\mathfrak m_{\varepsilon}]^T.
 \end{displaymath}
 Since
 \begin{displaymath}
 \|\nabla f_j(u_1,\dots,u_T)\|^2 =
 4\sum_{t = 1}^{T}p_{j,t}^{2}[(\mathbf T^0 -\mathbf T)\mathbf\Lambda]_{j,t}^{2}
 \textrm{ $;$ }
 \forall u\in [0,\mathfrak m_{\varepsilon}]^T,
 \end{displaymath}
 by Lemma \ref{Bernstein_dependent_rv}:
 \begin{eqnarray*}
 \mathbb E\left[\exp\left(2\lambda\sum_{t = 1}^{T}
 p_{j,t}[(\mathbf T^0 -\mathbf T)\mathbf\Lambda]_{j,t}\varepsilon_{j,t}\right)\right]
 & = &
 \mathbb E(\exp(\lambda(f_j(\varepsilon_{j,1},\dots,\varepsilon_{j,T})
 -\mathbb{E}[f_j(\varepsilon_{j,1},\dots,\varepsilon_{j,T})])))\\
 & \leqslant &
 \exp\left(
 2\mathfrak m_{\varepsilon}^2\lambda^2\Phi_{\varepsilon}^{2}
 \sum_{t = 1}^{T}
 p_{j,t}^{2}[(\mathbf T^0 -\mathbf T)\mathbf\Lambda]_{j,t}^{2}\right).
 \end{eqnarray*}
 Thus, for any $\lambda > 0$, by Equalities (\ref{preliminary_risk_bound_3}) and (\ref{preliminary_risk_bound_5}) together with $n\leqslant dT$,
 \begin{align*}
 \mathbb E\left[\exp\left(
 \lambda\sum_{i = 1}^{n}E_i\right)\right]
 & =\prod_{j = 1}^{d}
 \mathbb E\left[\exp\left(
 2\lambda\sum_{t = 1}^{T}p_{j,t}[(\mathbf T^0 -\mathbf T)\mathbf\Lambda]_{j,t}\varepsilon_{j,t}\right)\right]\\
 & \leqslant
 \prod_{j = 1}^{d}\exp\left(
 2\mathfrak m_{\varepsilon}^2\lambda^2\Phi_{\varepsilon}^{2}
 \sum_{t = 1}^{T}
 p_{j,t}^{2}[(\mathbf T^0 -\mathbf T)\mathbf\Lambda]_{j,t}^{2}\right)\\
 & \leqslant\exp\left[
 \frac{2\mathfrak m_{\varepsilon}^2\lambda^2\Phi_{\varepsilon}^{2}\mathfrak c_{\Pi}}{dT}
 \sum_{j,t}
 p_{j,t}[(\mathbf T^0 -\mathbf T)\mathbf\Lambda]_{j,t}^{2}\right]
 \leqslant\exp\left[
 \frac{2\mathfrak m_{\varepsilon}^2\lambda^2\Phi_{\varepsilon}^{2}\mathfrak c_{\Pi}}{n}
 (R(\mathbf T^0\mathbf\Lambda) - R(\mathbf T\mathbf\Lambda))\right].
 \end{align*}
\end{itemize}
Therefore, these bounds together with Jensen's inequality give:
\begin{align*}
 \mathbb E\exp & \left(\frac{\lambda}{4}
 [r_n(\mathbf T\mathbf\Lambda) - r_n(\mathbf T^0\mathbf\Lambda) +
 R(\mathbf T^0\mathbf\Lambda) - R(\mathbf T\mathbf\Lambda)]\right)\\
 & = 
 \mathbb E\left[\exp\left(\frac{\lambda}{4}
 \sum_{i = 1}^{n}(h(\Xi_i) -\mathbb E(h(\Xi_i)))\right)\right]\\
 & = \mathbb E\left[\exp\left(\frac{\lambda}{4}
 \sum_{i = 1}^{n}A_i +\frac{\lambda}{4}
 \sum_{i = 1}^{n}(C_i -\mathbb E(h(\Xi_i))) +
 \frac{\lambda}{4}\sum_{i = 1}^{n}D_i +
 \frac{\lambda}{4}\sum_{i = 1}^{n}E_i\right)\right]\\
 & \leqslant\frac{1}{4}\left[
 \mathbb E\left[\exp\left(\lambda\sum_{i = 1}^{n}A_i\right)\right] +
 \mathbb E\left[\exp\left(\lambda\sum_{i = 1}^{n}(C_i -\mathbb E(h(\Xi_i)))\right)\right]\right.\\
 &
 \quad +\left.
 \mathbb E\left[\exp\left(\lambda\sum_{i = 1}^{n}D_i\right)\right] +
 \mathbb E\left[\exp\left(\lambda\sum_{i = 1}^{n}E_i\right)\right]\right]\\
 & \leqslant
 \exp\left[\frac{2\mathfrak v_{\xi}}{1 - 4\mathfrak c_{\xi}\mathfrak m_0\lambda/n}
 \cdot\frac{\lambda^2}{n}
 (R(\mathbf T^0\mathbf\Lambda) - R(\mathbf T\mathbf\Lambda))\right]
 +\exp\left[\frac{2\mathfrak m_{0}^{2}}{1 - 4\mathfrak m_{0}^{2}\lambda/n}
 \cdot\frac{\lambda^2}{n}
 (R(\mathbf T^0\mathbf\Lambda) - R(\mathbf T\mathbf\Lambda))\right]\\
 & \quad
 +\exp\left[\frac{2\mathfrak m_{\varepsilon}^{2}}{1 - 8\mathfrak m_{\varepsilon}\mathfrak m_0\lambda/n}
 \cdot\frac{\lambda^2}{n}
 (R(\mathbf T^0\mathbf\Lambda) - R(\mathbf T\mathbf\Lambda))\right] +
 \exp\left[2\mathfrak m_{\varepsilon}^2\Phi_{\varepsilon}^{2}\mathfrak c_{\Pi}
 \frac{\lambda^2}{n}
 (R(\mathbf T^0\mathbf\Lambda) - R(\mathbf T\mathbf\Lambda))\right]\\
 & \leqslant
 \exp\left[\mathfrak c_\lambda
 \frac{\lambda^2}{n}(R(\mathbf T^0\mathbf\Lambda) - R(\mathbf T\mathbf\Lambda))\right]
\end{align*}
with
\begin{displaymath}
\mathfrak c_{\lambda} =
\max\left\{
\frac{2\mathfrak v_{\xi}}{1 - 4\mathfrak c_{\xi}\mathfrak m_0\lambda/n},
\frac{2\mathfrak m_{0}^{2}}{1 - 4\mathfrak m_{0}^{2}\lambda/n},
\frac{2\mathfrak m_{\varepsilon}^{2}}{1 - 8\mathfrak m_{\varepsilon}\mathfrak m_0\lambda/n},
2\mathfrak m_{\varepsilon}^2\Phi_{\varepsilon}^{2}\mathfrak c_{\Pi}\right\}
\end{displaymath}
and
\begin{displaymath}
0 <\lambda <
n\min\left\{
\frac{1}{4\mathfrak c_{\xi}\mathfrak m_0},
\frac{1}{4\mathfrak m_{0}^{2}},
\frac{1}{8\mathfrak m_{\varepsilon}\mathfrak m_0}\right\}.
\end{displaymath}
In particular, for
\begin{displaymath}
\lambda <
\frac{n}{16\mathfrak m_0
\max\{\mathfrak m_0,\mathfrak m_{\varepsilon},
\mathfrak c_{\xi}\}},
\end{displaymath}
we have
\begin{displaymath}
\mathfrak c_{\lambda}
\leqslant
\max\{4\mathfrak m_{0}^{2},
4\mathfrak v_{\xi},
4\mathfrak m_{\varepsilon}^{2},
2\mathfrak m_{\varepsilon}^{2}\Phi_{\varepsilon}^{2}\mathfrak c_{\Pi}\}.
\end{displaymath}
This ends the proof of the first inequality.
\end{proof}
%


%
\subsection{A preliminary non-explicit risk bound}
We now provide a simpler version of Theorem~\ref{explicit_risk_bound}, that holds in the case where $\mathcal S_{k,\tau}$ is finite: (1) in the following theorem. When this is not the case, we provide a similar bound using a general $\epsilon$-net, that is (2) in the theorem.

\begin{theorem}\label{risk_bound}
Consider $\alpha\in ]0,1[$.
\begin{enumerate}
 \item Under Assumptions \ref{assumption_phi_mixing}, \ref{assumption_boundedness} and \ref{assumption_probabilities}, if $|\mathcal S_{k,\tau}| <\infty$, then
 \begin{displaymath}
 \|\widehat{\mathbf\Theta}_{k,\tau} -\mathbf\Theta^0\|_{\mathcal F,\Pi}^{2}
 \leqslant
 3\min_{\mathbf T\in\mathcal S_{k,\tau}}
 \|(\mathbf T -\mathbf T^0)\mathbf\Lambda\|_{\mathcal F,\Pi}^{2}
 +\frac{\mathfrak c_{\ref{risk_bound},1}}{n}
 \log\left(\frac{2}{\alpha}|\mathcal S_{k,\tau}|\right)
 \end{displaymath}
 with probability larger than $1 -\alpha$, where $\mathfrak c_{\ref{risk_bound},1} = 32
 (\mathfrak c_{\ref{preliminary_risk_bound}}^{-1}\wedge\lambda^*)^{-1}$.
 \item Under Assumptions \ref{assumption_phi_mixing}, \ref{assumption_boundedness} and \ref{assumption_probabilities}, for every $\epsilon > 0$, there exists a finite subset $\mathcal S_{k,\tau}^{\epsilon}$ of $\mathcal S_{k,\tau}$ such that
 \begin{displaymath}
 \|\widehat{\mathbf\Theta}_{k,\tau} -\mathbf\Theta^0\|_{\mathcal F,\Pi}^{2}
 \leqslant
 3\min_{\mathbf T\in\mathcal S_{k,\tau}}
 \|(\mathbf T -\mathbf T^0)\mathbf\Lambda\|_{\mathcal F,\Pi}^2 +
 \frac{\mathfrak c_{\ref{risk_bound},1}}{n}
 \log\left(\frac{2}{\alpha}|\mathcal S_{k,\tau}^{\epsilon}|\right) +
 \left[ \mathfrak c_{\ref{risk_bound},2}
 + 8\mathfrak m_{\mathbf\Lambda}\mathfrak c_{\xi}
 \log\left(\frac{1}{\alpha}\right)
 \right]\epsilon
 \end{displaymath}
 with probability larger than $1 -\alpha$, where $\mathfrak c_{\ref{risk_bound},2} = 4\mathfrak m_{\mathbf\Lambda}(\mathfrak v_{\xi}^{1/2} +\mathfrak v_{\xi}/(2\mathfrak c_{\xi}) +\mathfrak m_{\varepsilon} + 3\mathfrak m_0)$.
\end{enumerate}
\end{theorem}
\begin{proof}[Proof of Theorem~\ref{risk_bound}]
\begin{enumerate}
 \item Assume that $|\mathcal S_{k,\tau}| <\infty$. For any $x > 0$, $\lambda\in (0,n\lambda^*)$ and $\mathcal S\subset\mathcal M_{d,\tau}(\mathbb R)$, consider the events
 \begin{displaymath}
 \Omega_{x,\lambda,\mathcal S}^{-}(\mathbf T) :=
 \left\{
 \left(1 -\mathfrak c_{\ref{preliminary_risk_bound}}\frac{\lambda}{n}\right)\|
 (\mathbf T -\mathbf T^0)\mathbf\Lambda\|_{\mathcal F,\Pi}^{2}
 -(r_n(\mathbf T\mathbf\Lambda) - r_n(\mathbf T^0\mathbf\Lambda)) > 4x
 \right\}
 \textrm{, }
 \mathbf T\in\mathcal S
 \end{displaymath}
 and
 \begin{displaymath}
 \Omega_{x,\lambda,\mathcal S}^{-} :=
 \bigcup_{\mathbf T\in\mathcal S}\Omega_{x,\lambda,\mathcal S}^{-}(\mathbf T).
 \end{displaymath}
 By Markov's inequality together with Proposition \ref{preliminary_risk_bound}, Inequality (\ref{preliminary_risk_bound_2}),
 \begin{eqnarray*}
  \mathbb P(\Omega_{x,\lambda,\mathcal S_{k,\tau}}^{-})
  & \leqslant &
  \sum_{\mathbf T\in\mathcal S_{k,\tau}}
  \mathbb P\left(
  \exp\left(\frac{\lambda}{4}\left(
  \left(1 -\mathfrak c_{\ref{preliminary_risk_bound}}\frac{\lambda}{n}\right)
  (R(\mathbf T\mathbf\Lambda) - R(\mathbf T^0\mathbf\Lambda))
  -(r_n(\mathbf T\mathbf\Lambda) - r_n(\mathbf T^0\mathbf\Lambda))\right)\right) >
  e^{\lambda x}\right)\\
  & \leqslant &
  |\mathcal S_{k,\tau}|e^{-\lambda x}.
 \end{eqnarray*}
 In the same way, with
 \begin{displaymath}
 \Omega_{x,\lambda,\mathcal S}^{+}(\mathbf T) :=
 \left\{
 -\left(1 +\mathfrak c_{\ref{preliminary_risk_bound}}\frac{\lambda}{n}\right)\|
 (\mathbf T -\mathbf T^0)\mathbf\Lambda\|_{\mathcal F,\Pi}^{2}
 + r_n(\mathbf T\mathbf\Lambda) - r_n(\mathbf T^0\mathbf\Lambda) > 4x\right\}
 \textrm{, }\mathbf T\in\mathcal S
 \end{displaymath}
 and
 \begin{displaymath}
 \Omega_{x,\lambda,\mathcal S}^{+} :=
 \bigcup_{\mathbf T\in\mathcal S}\Omega_{x,\lambda,\mathcal S}^{+}(\mathbf T),
 \end{displaymath}
 by Markov's inequality together with Proposition \ref{preliminary_risk_bound}, Inequality (\ref{preliminary_risk_bound_1}), $\mathbb P(\Omega_{x,\lambda,\mathcal S_{k,\tau}}^{+})\leqslant |\mathcal S_{k,\tau}|e^{-\lambda x}$. Then,
 \begin{displaymath}
 \mathbb P(\Omega_{x,\lambda,\mathcal S_{k,\tau}})\geqslant
 1 - 2|\mathcal S_{k,\tau}|e^{-\lambda x}
 \end{displaymath}
 with
 \begin{displaymath}
 \Omega_{x,\lambda,\mathcal S} := (\Omega_{x,\lambda,\mathcal S}^{-})^c\cap (\Omega_{x,\lambda,\mathcal S}^{+})^c\subset
 \Omega_{x,\lambda,\mathcal S}^{-}(\widehat{\mathbf T}_{k,\tau})^c\cap
 \Omega_{x,\lambda,\mathcal S}^{+}(\widehat{\mathbf T}_{k,\tau})^c
 =:\Omega_{x,\lambda,\mathcal S_{k,\tau}}(\widehat{\mathbf T}_{k,\tau}).
 \end{displaymath}
 Moreover, on the event $\Omega_{x,\lambda,\mathcal S_{k,\tau}}$, by the definition of $\widehat{\mathbf T}_{k,\tau}$,
 \begin{eqnarray*}
  \|\widehat{\mathbf\Theta}_{k,\tau} -\mathbf\Theta^0\|_{\mathcal F,\Pi}^{2}
  & \leqslant &
  \left(1 -\mathfrak c_{\ref{preliminary_risk_bound}}\frac{\lambda}{n}\right)^{-1}
  (r_n(\widehat{\mathbf T}_{k,\tau}\mathbf\Lambda) - r_n(\mathbf T^0\mathbf\Lambda) + 4x)\\
  & = &
  \left(1 -\mathfrak c_{\ref{preliminary_risk_bound}}\frac{\lambda}{n}\right)^{-1}
  \left(
  \min_{\mathbf T\in\mathcal S_{k,\tau}}
  \{r_n(\mathbf T\mathbf\Lambda) - r_n(\mathbf T^0\mathbf\Lambda)\} + 4x\right)\\
  & \leqslant &
  \frac{1 +\mathfrak c_{\ref{preliminary_risk_bound}}\lambda n^{-1}}{
  1 -\mathfrak c_{\ref{preliminary_risk_bound}}\lambda n^{-1}}
  \min_{\mathbf T\in\mathcal S_{k,\tau}}
  \|(\mathbf T -\mathbf T^0)\mathbf\Lambda\|_{\mathcal F,\Pi}^{2}
  +\frac{8x}{1 -\mathfrak c_{\ref{preliminary_risk_bound}}\lambda n^{-1}}.
 \end{eqnarray*}
 So, for any $\alpha\in ]0,1[$, with probability larger than $1 -\alpha$,
 \begin{displaymath}
 \|\widehat{\mathbf\Theta}_{k,\tau} -\mathbf\Theta^0\|_{\mathcal F,\Pi}^{2}
 \leqslant
 \frac{1 +\mathfrak c_{\ref{preliminary_risk_bound}}\lambda n^{-1}}{
 1 -\mathfrak c_{\ref{preliminary_risk_bound}}\lambda n^{-1}}
 \min_{\mathbf T\in\mathcal S_{k,\tau}}
 \|(\mathbf T -\mathbf T^0)\mathbf\Lambda\|_{\mathcal F,\Pi}^{2}
 +\frac{8\lambda^{-1}\log(2\alpha^{-1}|\mathcal S_{k,\tau}|)}{1 -\mathfrak c_{\ref{preliminary_risk_bound}}\lambda n^{-1}}.
 \end{displaymath}
 Now, let us take
 \begin{displaymath}
 \lambda =
 \frac{n}{2}\left(\frac{1}{\mathfrak c_{\ref{preliminary_risk_bound}}}\wedge\lambda^*\right)
 \in (0,n\lambda^*)
 \textrm{ and }
 x =
 \frac{1}{\lambda}\log\left(\frac{2}{\alpha}|\mathcal S_{k,\tau}|\right).
 \end{displaymath}
 In particular, $\mathfrak c_{\ref{preliminary_risk_bound}}\lambda n^{-1}\leqslant 1/2$, and then
 \begin{displaymath}
 \frac{1 +\mathfrak c_{\ref{preliminary_risk_bound}}\lambda n^{-1}}{
 1 -\mathfrak c_{\ref{preliminary_risk_bound}}\lambda n^{-1}}
 \leqslant 3
 \textrm{ and }
 \frac{8\lambda^{-1}}{
 1 -\mathfrak c_{\ref{preliminary_risk_bound}}\lambda n^{-1}}
 \leqslant 32
 \left(\frac{1}{\mathfrak c_{\ref{preliminary_risk_bound}}}\wedge\lambda^*\right)^{-1}\frac{1}{n}.
 \end{displaymath}
 Therefore, with probability larger than $1 -\alpha$,
 \begin{displaymath}
 \|\widehat{\mathbf\Theta}_{k,\tau} -\mathbf\Theta^0\|_{\mathcal F,\Pi}^{2}
 \leqslant
 3\min_{\mathbf T\in\mathcal S_{k,\tau}}
 \|(\mathbf T -\mathbf T^0)\mathbf\Lambda\|_{\mathcal F,\Pi}^{2}
 + 32\left(\frac{1}{\mathfrak c_{\ref{preliminary_risk_bound}}}\wedge\lambda^*\right)^{-1}
 \frac{1}{n}
 \log\left(\frac{2}{\alpha}|\mathcal S_{k,\tau}|\right).
 \end{displaymath}
 \item Now, assume that $|\mathcal S_{k,\tau}| =\infty$. Since $\textrm{dim}(\mathcal M_{d,\tau}(\mathbb R)) <\infty$ and $\mathcal S_{k,\tau}$ is a bounded subset of $\mathcal M_{d,\tau}(\mathbb R)$ (equipped with $\mathbf T\mapsto\sup_{j,t}|\mathbf T_{j,t}|$), $\mathcal S_{k,\tau}$ is compact in $(\mathcal M_{d,\tau}(\mathbb R),\|.\|_{\mathcal F})$. Then, for any $\epsilon > 0$, there exists a finite subset $\mathcal S_{k,\tau}^{\epsilon}$ of $\mathcal S_{k,\tau}$ such that
 \begin{equation}\label{risk_bound_1}
 \forall\mathbf T\in\mathcal S_{k,\tau},
 \exists\mathbf T^{\epsilon}\in\mathcal S_{k,\tau}^{\epsilon} :
 \|\mathbf T -\mathbf T^{\epsilon}\|_{\mathcal F}\leqslant\epsilon.
 \end{equation}
 On the one hand, for any $\mathbf T\in\mathcal S_{k,\tau}$ and $\mathbf T^{\epsilon}\in\mathcal S_{k,\tau}^{\epsilon}$ satisfying (\ref{risk_bound_1}), since $\langle\mathbf X_i,(\mathbf T -\mathbf T^{\epsilon})\mathbf\Lambda\rangle_{\mathcal F} =\langle\mathbf X_i\mathbf\Lambda^*,\mathbf T -\mathbf T^{\epsilon}\rangle_{\mathcal F}$ for every $i\in\{1,\dots,n\}$,
 \begin{eqnarray}
  |r_n(\mathbf T\mathbf\Lambda) - r_n(\mathbf T^{\epsilon}\mathbf\Lambda)|
  & \leqslant &
  \frac{1}{n}\sum_{i = 1}^{n}
  |\langle\mathbf X_i,(\mathbf T -\mathbf T^{\epsilon})\mathbf\Lambda\rangle_{\mathcal F}
  (2Y_i -\langle\mathbf X_i,(\mathbf T +\mathbf T^{\epsilon})\mathbf\Lambda\rangle_{\mathcal F})|
  \nonumber\\
  & \leqslant &
  \frac{\epsilon}{n}\sum_{i = 1}^{n}
  \|\mathbf X_i\mathbf\Lambda^*\|_{\mathcal F}\left(
  2|Y_i| +\sup_{j,t}\left|\sum_{\ell = 1}^{\tau}(\mathbf T +\mathbf T^{\epsilon})_{j,\ell}
  \mathbf\Lambda_{\ell,t}\right|\right)
  \nonumber\\
  \label{risk_bound_2}
  & \leqslant &
  \epsilon\mathfrak m_{\mathbf\Lambda}
  \left(\frac{2}{n}\sum_{i = 1}^{n}|Y_i| + 2\mathfrak m_0\right)
  \leqslant\mathfrak c_1(\xi_1,\dots,\xi_n)\epsilon
 \end{eqnarray}
 with
 \begin{displaymath}
 \mathfrak c_1(\xi_1,\dots,\xi_n) :=
 2\mathfrak m_{\mathbf\Lambda}
 \left(\frac{1}{n}\sum_{i = 1}^{n}|\xi_i| +
 \mathfrak m_{\varepsilon} + 2\mathfrak m_0\right),
 \end{displaymath}
 and thanks to Equality (\ref{preliminary_risk_bound_3}),
 \begin{eqnarray}
  |R(\mathbf T\mathbf\Lambda) - R(\mathbf T^{\epsilon}\mathbf\Lambda)|
  & = &
  |R(\mathbf T\mathbf\Lambda) - R(\mathbf T^0\mathbf\Lambda)
  -(R(\mathbf T^{\epsilon}\mathbf\Lambda) - R(\mathbf T^0\mathbf\Lambda))|
  \nonumber\\
  & = &
  |\|(\mathbf T -\mathbf T^0)\mathbf\Lambda\|_{\mathcal F,\Pi}^{2} -
  \|(\mathbf T^{\epsilon} -\mathbf T^0)\mathbf\Lambda\|_{\mathcal F,\Pi}^{2}|
  \nonumber\\
  \label{risk_bound_3}
  & \leqslant &
  \mathbb E(|\langle\mathbf X_i,(\mathbf T -\mathbf T^{\epsilon})\mathbf\Lambda\rangle_{\mathcal F}
  \langle\mathbf X_i,(\mathbf T +\mathbf T^{\epsilon} - 2\mathbf T^0)\mathbf\Lambda\rangle_{\mathcal F}|)
  \leqslant\mathfrak c_2\epsilon
 \end{eqnarray}
 with $\mathfrak c_2 = 4\mathfrak m_0\mathfrak m_{\mathbf\Lambda}$. On the other hand, consider
 \begin{equation}\label{risk_bound_4}
 \widehat{\mathbf T}_{k,\tau}^{\epsilon}
 =\arg\min_{\mathbf T\in\mathcal S_{k,\tau}^{\epsilon}}
 \|\mathbf T -\widehat{\mathbf T}_{k,\tau}\|_{\mathcal F}.
 \end{equation}
 On the event $\Omega_{x,\lambda,\mathcal S_{k,\tau}^{\epsilon}}$ with $x > 0$ and $\lambda\in (0,n\lambda^*)$, by the definitions of $\widehat{\mathbf T}_{k,\tau}^{\epsilon}$ and $\widehat{\mathbf T}_{k,\tau}$, and thanks to Inequalities (\ref{risk_bound_2}) and (\ref{risk_bound_3}),
 \begin{eqnarray*}
  \|\widehat{\mathbf\Theta}_{k,\tau} -\mathbf\Theta^0\|_{\mathcal F,\Pi}^{2}
  & \leqslant &
  \|(\widehat{\mathbf T}_{k,\tau}^{\epsilon} -\mathbf T^0)\mathbf\Lambda\|_{\mathcal F,\Pi}^{2} +
  \mathfrak c_2\epsilon
  \leqslant  
  \left(1 -\mathfrak c_{\ref{preliminary_risk_bound}}\frac{\lambda}{n}\right)^{-1}
  (r_n(\widehat{\mathbf T}_{k,\tau}^{\epsilon}\mathbf\Lambda) - r_n(\mathbf T^0\mathbf\Lambda) + 4x) +
  \mathfrak c_2\epsilon\\
  & \leqslant &
  \left(1 -\mathfrak c_{\ref{preliminary_risk_bound}}\frac{\lambda}{n}\right)^{-1}
  [r_n(\widehat{\mathbf T}_{k,\tau}\mathbf\Lambda) - r_n(\mathbf T^0\mathbf\Lambda) 
  +\mathfrak c_1(\xi_1,\dots,\xi_n)\epsilon + 4x] +
  \mathfrak c_2\epsilon\\
  & = &
  \left(1 -\mathfrak c_{\ref{preliminary_risk_bound}}\frac{\lambda}{n}\right)^{-1}
  \left[
  \min_{\mathbf T\in\mathcal S_{k,\tau}}
  \{r_n(\mathbf T\mathbf\Lambda) - r_n(\mathbf T^0\mathbf\Lambda)\} +
  \mathfrak c_1(\xi_1,\dots,\xi_n)\epsilon + 4x\right]
  +\mathfrak c_2\epsilon\\
  & \leqslant &
  \frac{1 +\mathfrak c_{\ref{preliminary_risk_bound}}\lambda n^{-1}}{
  1 -\mathfrak c_{\ref{preliminary_risk_bound}}\lambda n^{-1}}
  \min_{\mathbf T\in\mathcal S_{k,\tau}}
  \|(\mathbf T -\mathbf T^0)\mathbf\Lambda\|_{\mathcal F,\Pi}^{2}
  +\frac{8x}{1 -\mathfrak c_{\ref{preliminary_risk_bound}}\lambda n^{-1}} +
  \left[\frac{\mathfrak c_1(\xi_1,\dots,\xi_n)}{1 -\mathfrak c_{\ref{preliminary_risk_bound}}\lambda n^{-1}}
  +\mathfrak c_2\right]\epsilon.
 \end{eqnarray*}
 So, by taking
 \begin{displaymath}
 \lambda =
 \frac{n}{2}\left(\frac{1}{\mathfrak c_{\ref{preliminary_risk_bound}}}\wedge\lambda^*\right)
 \textrm{ and }
 x =
 \frac{1}{\lambda}\log\left(\frac{2}{\alpha}|\mathcal S_{k,\tau}^{\epsilon}|\right),
 \end{displaymath}
 as in the proof of Theorem \ref{risk_bound}.(1), with probability larger than $1 -\alpha$,
 \begin{eqnarray}
  \label{risk_bound_5}
  \|\widehat{\mathbf\Theta}_{k,\tau} -\mathbf\Theta^0\|_{\mathcal F,\Pi}^{2}
  & \leqslant &
  3\min_{\mathbf T\in\mathcal S_{k,\tau}}
  \|(\mathbf T -\mathbf T^0)\mathbf\Lambda\|_{\mathcal F,\Pi}^2 +
  32\left(\frac{1}{\mathfrak c_{\ref{preliminary_risk_bound}}}\wedge\lambda^*\right)^{-1}
  \frac{1}{n}
  \log\left(\frac{2}{\alpha}|\mathcal S_{k,\tau}^{\epsilon}|\right)\\
  & & \hspace{4cm} +
  \left[4\mathfrak m_{\mathbf\Lambda}\left(\frac{1}{n}\sum_{i = 1}^{n}|\xi_i| +
  \mathfrak m_{\varepsilon} + 2\mathfrak m_0\right) +\mathfrak c_2\right]\epsilon.
  \nonumber
 \end{eqnarray}
 Thanks to Markov's inequality together with Lemma~\ref{Bernstein_independent_rv}, for $\lambda_0 = 1/(2n\mathfrak c_{\xi})$,
 \begin{eqnarray*}
  \mathbb P\left(\sum_{i = 1}^{n}|\xi_i| >
  \sum_{i = 1}^{n}\mathbb E(|\xi_i|) + s\right)
  & \leqslant &
  \exp\left[\frac{n\mathfrak v_{\xi}\lambda_{0}^{2}}{2(1 - n\mathfrak c_{\xi}\lambda_0)}
  -\lambda_0s\right]\\
  & = &
  \exp\left(\frac{\mathfrak v_{\xi}}{4n\mathfrak c_{\xi}^{2}} -
  \frac{s}{2n\mathfrak c_{\xi}}\right) =\alpha
 \end{eqnarray*}
 with
 \begin{displaymath}
 s =\frac{\mathfrak v_{\xi}}{2\mathfrak c_{\xi}} +
 2n\mathfrak c_{\xi}\log\left(\frac{1}{\alpha}\right).
 \end{displaymath}
 Then, since $\mathbb E(|\xi_i|)\leqslant\mathbb E(\xi_{i}^{2})^{1/2}\leqslant\mathfrak v_{\xi}^{1/2}$ for every $i\in\{1,\dots,n\}$,
 \begin{equation}\label{risk_bound_6}
 \mathbb P\left[\frac{1}{n}\sum_{i = 1}^{n}|\xi_i| >
 \mathfrak v_{\xi}^{1/2} +\frac{\mathfrak v_{\xi}}{2n\mathfrak c_{\xi}} +
 2\mathfrak c_{\xi}\log\left(\frac{1}{\alpha}\right)\right]\leqslant\alpha.
 \end{equation}
 Finally, note that if $\mathbb P(U > V + c)\leqslant\alpha$ and $\mathbb P(V > v)\leqslant\alpha$ with $c,v\in\mathbb R_+$ and $(U,V)$ a $\mathbb R^2$-valued random variable, then
 \begin{eqnarray}
  \mathbb P(U > v + c) & = &
  \mathbb P(U > v + c,V > v) +\mathbb P(U > v + c,V\leqslant v)
  \nonumber\\
  \label{risk_bound_7}
  & \leqslant &
  \mathbb P(V > v) +\mathbb P(U > V + c,V\leqslant v)
  \leqslant 2\alpha.
 \end{eqnarray}
 Therefore, by (\ref{risk_bound_5}) and (\ref{risk_bound_6}), with probability larger than $1 - 2\alpha$,
 \begin{eqnarray*}
  \|\widehat{\mathbf\Theta}_{k,\tau} -\mathbf\Theta^0\|_{\mathcal F,\Pi}^{2}
  & \leqslant &
  3\min_{\mathbf T\in\mathcal S_{k,\tau}}
  \|(\mathbf T -\mathbf T^0)\mathbf\Lambda\|_{\mathcal F,\Pi}^2 +
  32\left(\frac{1}{\mathfrak c_{\ref{preliminary_risk_bound}}}\wedge\lambda^*\right)^{-1}
  \frac{1}{n}
  \log\left(\frac{2}{\alpha}|\mathcal S_{k,\tau}^{\epsilon}|\right)\\
  & &
  \hspace{3cm}
  +\left[4\mathfrak m_{\mathbf\Lambda}\left(
  2\mathfrak c_{\xi}\log\left(\frac{1}{\alpha}\right) +
  \mathfrak v_{\xi}^{1/2} +
  \frac{\mathfrak v_{\xi}}{2\mathfrak c_{\xi}} +\mathfrak m_{\varepsilon} + 2\mathfrak m_0\right) +\mathfrak c_2\right]\epsilon.
\end{eqnarray*}
\end{enumerate}
\end{proof}
%


%
\subsection{Proof of Theorem \ref{explicit_risk_bound}}
The proof is dissected in two steps:
\\
\\
\textbf{Step 1.} Consider
\begin{displaymath}
\mathcal M_{d,\tau,k}(\mathbb R) :=
\{\mathbf T\in\mathcal M_{d,\tau}(\mathbb R) :\textrm{rank}(\mathbf T) = k\}.
\end{displaymath}
For every $\mathbf T\in\mathcal M_{d,\tau,k}(\mathbb R)$ and $\rho > 0$, let us denote the closed ball (resp. the sphere) of center $\mathbf T$ and of radius $\rho$ of $\mathcal M_{d,\tau,k}(\mathbb R)$ by $\mathbb B_k(\mathbf T,\rho)$ (resp. $\mathbb S_k(\mathbf T,\rho)$). For any $\epsilon > 0$, thanks to Cand\`es and Plan \cite{CP11}, Lemma 3.1, there exists an $\epsilon$-net $\mathbb S_{k}^{\epsilon}(0,1)$ covering $\mathbb S_k(0,1)$ and such that
\begin{displaymath}
|\mathbb S_{k}^{\epsilon}(0,1)|\leqslant
\left(\frac{9}{\epsilon}\right)^{k(d +\tau + 1)}.
\end{displaymath}
Then, for every $\rho > 0$, there exists an $\epsilon$-net $\mathbb S_{k}^{\epsilon}(0,\rho)$ covering $\mathbb S_k(0,\rho)$ and such that
\begin{displaymath}
|\mathbb S_{k}^{\epsilon}(0,\rho)|\leqslant
\left(\frac{9\rho}{\epsilon}\right)^{k(d +\tau + 1)}.
\end{displaymath}
Moreover, for any $\rho^* > 0$,
\begin{displaymath}
\mathbb B_k(0,\rho^*) =
\bigcup_{\rho\in [0,\rho^*]}
\mathbb S_k(0,\rho).
\end{displaymath}
So,
\begin{displaymath}
\mathbb B_{k}^{\epsilon}(0,\rho^*) :=
\bigcup_{j = 0}^{[\rho^*/\epsilon] + 1}
\mathbb S_{k}^{\epsilon}(0,j\epsilon)
\end{displaymath}
is an $\epsilon$-net covering $\mathbb B_k(0,\rho^*)$ and such that
\begin{displaymath}
|\mathbb B_{k}^{\epsilon}(0,\rho^*)|
\leqslant
\sum_{j = 0}^{[\rho^*/\epsilon] + 1}
|\mathbb S_{k}^{\epsilon}(0,j\epsilon)|
\leqslant
\left(\left[\frac{\rho^*}{\epsilon}\right] + 2\right)
\left(\frac{9\rho^*}{\epsilon}\right)^{k(d +\tau + 1)}.
\end{displaymath}
If in addition $\rho^*\geqslant\epsilon$, then
\begin{displaymath}
|\mathbb B_{k}^{\epsilon}(0,\rho^*)|
\leqslant
\frac{3\rho^*}{\epsilon}\left(\frac{9\rho^*}{\epsilon}\right)^{k(d +\tau + 1)}
\leqslant
\left(\frac{9\rho^*}{\epsilon}\right)^{2k(d +\tau)}.
\end{displaymath}
\textbf{Step 2.} For any $\mathbf T\in\mathcal S_{k,\tau}$,
\begin{displaymath}
\sup_{j,t}|\mathbf T_{j,t}|
\leqslant\frac{\mathfrak m_0}{\mathfrak m_{\mathbf\Lambda}(\tau)}.
\end{displaymath}
Then,
\begin{displaymath}
\|\mathbf T\|_{\mathcal F}
=\left(\sum_{j = 1}^{d}\sum_{t = 1}^{\tau}\mathbf T_{j,t}^{2}\right)^{1/2}
\leqslant
\rho_{d,\tau}^{*} :=
\mathfrak m_0\frac{d^{1/2}\tau^{1/2}}{\mathfrak m_{\bf\Lambda}(\tau)}.
\end{displaymath}
So, $\mathcal S_{k,\tau}\subset\mathbb B_k(0,\rho_{d,\tau}^{*})$, and by the first step of the proof, there exists an $\epsilon$-net $\mathcal S_{k,\tau}^{\epsilon}$ covering $\mathcal S_{k,\tau}$ and such that
\begin{displaymath}
|\mathcal S_{k,\tau}^{\epsilon}|
\leqslant
\left(\frac{9\rho_{d,\tau}^{*}}{\epsilon}\right)^{2k(d +\tau)}
=\left(9\mathfrak m_0\frac{d^{1/2}\tau^{1/2}}
{\mathfrak m_{\bf\Lambda}(\tau)\epsilon}\right)^{2k(d +\tau)}.
\end{displaymath}
By taking $\epsilon = 9\mathfrak m_0d^{1/2}\tau^{1/2}\mathfrak m_{\bf\Lambda}(\tau)^{-1}n^{-2}$, thanks to Theorem \ref{risk_bound}.(2), with probability larger than $1 -\alpha$,
\begin{eqnarray*}
 \|\widehat{\mathbf\Theta}_{k,\tau} -\mathbf\Theta^0\|_{\mathcal F,\Pi}^{2}
 & \leqslant &
 3\min_{\mathbf T\in\mathcal S_{k,\tau}}
 \|(\mathbf T -\mathbf T^0)\mathbf\Lambda\|_{\mathcal F,\Pi}^2\\
 & & +
 \frac{\mathfrak c_{\ref{risk_bound},1}}{n}\left[\log\left(\frac{2}{\alpha}\right) +
 2k(d +\tau)
 \log\left(9\mathfrak m_0\frac{d^{1/2}\tau^{1/2}}{\mathfrak m_{\bf\Lambda}(\tau)\epsilon}\right)\right] +
 \left[\mathfrak c_{\ref{risk_bound},2} +
 8\mathfrak m_{\mathbf \Lambda}\mathfrak c_{\xi}\log\left(\frac{1}{\alpha}\right)\right]
 \epsilon\\
 & = &
 3\min_{\mathbf T\in\mathcal S_{k,\tau}}
 \|(\mathbf T -\mathbf T^0)\mathbf\Lambda\|_{\mathcal F,\Pi}^2\\
 & & +
 \frac{\mathfrak c_{\ref{risk_bound},1}}{n}\left[\log\left(\frac{2}{\alpha}\right) +
 4k(d +\tau)\log(n)\right] +
 9 \mathfrak m_0\frac{d^{1/2}\tau^{1/2}}{\mathfrak m_{\bf\Lambda}(\tau)n^2}
 \left[\mathfrak c_{\ref{risk_bound},2} +
 8\mathfrak m_{\mathbf\Lambda}\mathfrak c_{\xi}\log\left(\frac{1}{\alpha}\right)\right].
\end{eqnarray*}
Therefore, since $n\geqslant\max(d,\tau)$ and $\mathfrak m_{\bf\Lambda}(\tau)\geqslant 1$, with probability larger than $1 - 2\alpha$,
\begin{eqnarray*}
 \|\widehat{\mathbf\Theta}_{k,\tau} -\mathbf\Theta^0\|_{\mathcal F,\Pi}^{2}
 & \leqslant &
 3\min_{\mathbf T\in\mathcal S_{k,\tau}}
 \|(\mathbf T -\mathbf T^0)\mathbf\Lambda\|_{\mathcal F,\Pi}^2\\
 & &
 + (4\mathfrak c_{\ref{risk_bound},1} +
 9\mathfrak m_0c_{\ref{risk_bound},2})k(d +\tau)\frac{\log(n)}{n} 
 + \frac{\mathfrak c_{\ref{risk_bound},1} +
 72\mathfrak m_0\mathfrak m_{\mathbf\Lambda}\mathfrak c_{\xi}}{n}
 \log\left(\frac{2}{\alpha}\right).
\end{eqnarray*}
Let us replace $\alpha$ by $\alpha/2$ to end the proof.
%


%
\subsection{Proof of Theorem~\ref{lower_bound}}
Put $\overline k = 2^{\lfloor\log_2(k)\rfloor}$, and note that $k/2\leqslant\overline k\leqslant k$. Fix $a > 0$ and define the set of matrices
\begin{displaymath}
\mathcal A =
\left\{{\bf A} = ({\bf A}_{i,j})_{1\leqslant i\leqslant d\vee\tau,1\leqslant j\leqslant\overline k} :
{\bf A}_{i,j}\in\{0,a\}\right\}.
\end{displaymath}
By Varshamov-Gilbert bound, there is a finite subset $\mathcal B\subset\mathcal A$ with ${\rm card}(\mathcal B)\geqslant 2^{\frac{\overline k(d\vee\tau)}{8}} + 1$, $0\in\mathcal B$, and each pair ${\bf A}\neq {\bf A}'$ in $\mathcal B$ differ by at least $\overline k(d\vee\tau)$ coordinates. This implies
\begin{displaymath}
\|{\bf A} - {\bf A}'\|_{\mathcal F}^{2}
\geqslant
\frac{\overline k(d\vee\tau)}{8}a^2\geqslant\frac{k(d\vee\tau)}{16}a^2.
\end{displaymath}
For any ${\bf A}$, define by block $\overline{\bf A} = ({\bf A}|{\bf 0})$ of dimension $(d\vee\tau)\times k$ (so the ${\bf 0}$ has $k -\overline k$ columns). We then define $\widetilde{\bf A}$ of dimension $d\times\tau$. The construction differs depending on $d$ and $\tau$:
\begin{itemize}
 \item If $d\geqslant\tau$,
 \begin{displaymath}
 \widetilde{\bf A} = ({\bf A}|\dots|{\bf A}|{\bf 0}).
 \end{displaymath}
 \item If $d < \tau$,
 \begin{displaymath}
 \widetilde{\bf A} = ({\bf A}|\dots|{\bf A}|{\bf 0})^*.
 \end{displaymath}
\end{itemize}
Note that this is clearly inspired by the construction in the proof of Theorem 5 in~\cite{KLT11}, however, here, we have to take care that, for $a$ small enough, each $\widetilde{\bf A}\in\mathcal A$ is also in $\mathcal M_{d,k,\tau}$. In order to do so, we introduce the vectors in $\mathbb R^{\overline k}$:
\begin{eqnarray*}
 v[1] & = &
 \sqrt{\frac{1}{k}}(\underbrace{
 \begin{array}{rcl}
  1 & \dots & 1
 \end{array}}_{\overline k})^*,\\
 v[2] & = &
 \sqrt{\frac{1}{k}}(\underbrace{
 \begin{array}{rcl}
  1 & \dots & 1
 \end{array}}_{\overline k/2}|
 \underbrace{
 \begin{array}{rcl}
  -1 & \dots & -1
 \end{array}}_{\overline k/2})^*,\\
 & \vdots & \\
 v[\overline k] & = & \sqrt{\frac{1}{k}}
 (\underbrace{
 \begin{array}{rcccl}
  1 & -1 & \dots & 1 & -1
 \end{array}}_{\overline k})^*.
\end{eqnarray*}
Now, remark that for ${\bf A}\in\mathcal{A}$ we have
\begin{displaymath}
{\bf A} =\underbrace{\sqrt{ak}\left(
\begin{array}{c}
 v[1]^*\\
 \hline\vdots\\
 \hline
 v[\overline k]^*
\end{array}
\right)}_{\mathbf B}
\underbrace{
\left(\left.
\sum_{i = 1}^{n}v[i]\mathbf 1_{{\bf A}_{i,1}\neq 0}\right|
\dots\left|
\sum_{i = 1}^{n}v[i]\mathbf 1_{{\bf A}_{i,k}\neq 0}\right.
\right)
\sqrt{\frac{a}{k}}}_{\mathbf C}
\end{displaymath}
and under this decomposition, it is clear that the entries of $\mathbf B$ and $\mathbf C$ are in $[0,\sqrt a]$. Playing with blocks, this gives trivially to a decomposition $\widetilde{\bf A} = {\bf U}{\bf V}$ where ${\bf U}$ is $d\times k$, ${\bf V}$ is $k\times \tau$ and the entries of ${\bf U}$ and ${\bf V}$ are also in $[0,\sqrt{a }]$. In other words, $\widetilde{\bf A}\in\mathcal M_{d,k,\tau}$ holds as soon as $a\leqslant\mathfrak m_0/k$.
Now, let $\mathbb P_{\bf A}$ be the data-generating distribution when ${\bf\Theta}^0 =\widetilde{\bf A}$ for $\mathbf A\in\mathcal B$, and ${\rm KL}$ be the Kullback-Leibler divergence. We have
\begin{displaymath}
{\rm KL}(\mathbb{P}_{0},\mathbb{P}_{{\bf A}})
=\frac{n}{2}\|\widetilde{\bf A}{\bf\Lambda}\|_{\mathcal F,\Pi}^{2}
\leqslant\frac{n}{2}a^2.
\end{displaymath}
Thus, we look for $a$ such that the condition
\begin{displaymath}
\frac{n}{2}a^2
\leqslant
\alpha\log({\rm card}(\mathcal B) - 1) =
\frac{\alpha\overline k(d\vee\tau)}{8}
\end{displaymath}
is satisfied for a given $0 <\alpha < 1/8$. Fix $\alpha = 1/16$. As $k\leqslant\overline k/2$, it's easy to check that
\begin{displaymath}
a =
\frac{1}{8}\sqrt{\frac{k(d\vee\tau)}{2n}}
\end{displaymath}
satisfies the condition. Also, remember that $\widetilde{\bf A}\in\mathcal M_{d,k,\tau}$ if $a\leqslant\mathfrak m_0/k$, which adds the condition
\begin{displaymath}
k\leqslant
\left(\frac{256\mathfrak m_{0}^{2}n}{d\vee\tau}\right)^{1/3}.
\end{displaymath}
Theorem 2.5 in~\cite{TSYBAKOV09} then tells us that the rate is given by the minimal distance, for ${\bf A}\neq{\bf A}'$ in $\mathcal B$:
\begin{align*}
 \|\widetilde{\bf A}{\bf\Lambda} -\widetilde{\bf A}'{\bf\Lambda}\|_{\mathcal F,\Pi}^{2}
 & =
 \frac{1}{d\tau}\|\widetilde{\bf A} -\widetilde{\bf A}'\|_{\mathcal F}^{2} =
 \frac{1}{d\tau}\left\lfloor
 \frac{d\wedge\tau}{k}\right\rfloor\|
 \widetilde{\bf A} -\widetilde{\bf A}'\|_{\mathcal F}^{2}\\
 & \geqslant
 \frac{1}{d\tau}\left\lfloor
 \frac{d\wedge\tau}{k}\right\rfloor
 \frac{k(d\vee\tau)}{16}a^2\\
 & \geqslant
 \frac{a^2}{32}
 =\frac{k(d\vee\tau)}{4096n}.
\end{align*}
%


%
\subsection{Proof of Theorem \ref{risk_bound_adaptive_estimator}}
For any $k\in\mathcal K$, let $\mathcal S_{k}^{\epsilon} :=\mathcal S_{k,\tau}^{\epsilon}$ be the $\epsilon$-net introduced in the proof of Theorem \ref{explicit_risk_bound}, and recall that for $\epsilon = 9\mathfrak m_0d^{1/2}\tau^{1/2}\mathfrak m_{\mathbf\Lambda}(\tau)^{-1}n^{-2}$,
\begin{displaymath}
|\mathcal S_{k}^{\epsilon}|
\leqslant
\left(9\mathfrak m_0
\frac{d^{1/2}\tau^{1/2}}{\mathfrak m_{\mathbf\Lambda}(\tau)\epsilon}\right)^{2k(d +\tau)} =
n^{4k(d +\tau)}.
\end{displaymath}
Then, for $\alpha\in (0,1)$ and $x_{k,\epsilon} :=\lambda^{-1}\log(2\alpha^{-1}|\mathcal K|\cdot|\mathcal S_{k}^{\epsilon}|)$ with $\lambda = n\mathfrak c_{{\rm pen}}^{-1}\in (0,n\lambda^*)$,
\begin{eqnarray}
 4x_{k,\epsilon} -\textrm{pen}(k) & = &
 \frac{4\mathfrak c_{{\rm pen}}}{n}\log\left(\frac{2}{\alpha}|\mathcal K|\cdot|\mathcal S_{k}^{\epsilon}|\right)
 -16\mathfrak c_{{\rm pen}}\frac{\log(n)}{n}k(d +\tau)
 \nonumber\\
 & \leqslant &
 \frac{4\mathfrak c_{{\rm pen}}}{n}\left[4k(d +\tau)\log(n) +\log\left(\frac{2}{\alpha}|\mathcal K|\right)\right]
 -16\mathfrak c_{{\rm pen}}\frac{\log(n)}{n}k(d +\tau)
 \nonumber\\
 \label{risk_bound_adaptive_estimator_1}
 & \leqslant &
 \frac{4\mathfrak c_{{\rm pen}}}{n}\log\left(\frac{2}{\alpha}|\mathcal K|\right) =:\mathfrak m_n.
\end{eqnarray}
Now, consider the event $\Omega_{\lambda,\epsilon} := (\Omega_{\lambda,\epsilon}^{-})^c\cap (\Omega_{\lambda,\epsilon}^{+})^c$ with
\begin{displaymath}
\Omega_{\lambda,\epsilon}^{-} :=
\bigcup_{k\in\mathcal K}
\bigcup_{\mathbf T\in\mathcal S_{k}^{\epsilon}}\Omega_{x_{k,\epsilon},\lambda,\mathcal S_{k}^{\epsilon}}^{-}(\mathbf T)
\textrm{ and }
\Omega_{\lambda,\epsilon}^{+} :=
\bigcup_{k\in\mathcal K}
\bigcup_{\mathbf T\in\mathcal S_{k}^{\epsilon}}\Omega_{x_{k,\epsilon},\lambda,\mathcal S_{k}^{\epsilon}}^{+}(\mathbf T).
\end{displaymath}
So,
\begin{displaymath}
\mathbb P(\Omega_{\lambda,\epsilon}^{c})\leqslant
\sum_{k\in\mathcal K}\sum_{\mathbf T\in\mathcal S_{k}^{\epsilon}}
[\mathbb P(\Omega_{x_{k,\epsilon},\lambda,\mathcal S_{k}^{\epsilon}}^{-}(\mathbf T)) +
\mathbb P(\Omega_{x_{k,\epsilon},\lambda,\mathcal S_{k}^{\epsilon}}^{+}(\mathbf T))]\\
\leqslant
2\sum_{k\in\mathcal K}|\mathcal S_{k}^{\epsilon}|e^{-\lambda x_{k,\epsilon}} =
\alpha
\end{displaymath}
and $\Omega_{x_{\widehat k,\epsilon},\lambda,\mathcal S_{\widehat k}^{\epsilon}}(\widehat{\mathbf T}_{\widehat k}^{\epsilon})\subset\Omega_{\lambda,\epsilon}$, where $\widehat{\mathbf T}_{k}^{\epsilon}$ is a solution of the minimization problem (\ref{risk_bound_4}) for every $k\in\mathcal K$.
\\
\\
On the event $\Omega_{\lambda,\epsilon}$, by the definition of $\widehat k$, and thanks to Inequalities (\ref{risk_bound_2}), (\ref{risk_bound_3}) and (\ref{risk_bound_4}),
\begin{eqnarray}
 \|\widehat{\mathbf\Theta} -\mathbf\Theta^0\|_{\mathcal F,\Pi}^{2}
 & \leqslant &
 \|(\widehat{\mathbf T}_{\widehat k}^{\epsilon} -\mathbf T^0)\mathbf\Lambda\|_{\mathcal F,\Pi}^{2} +
 \mathfrak c_2\epsilon
 \leqslant  
 \left(1 -\mathfrak c_{\ref{preliminary_risk_bound}}\frac{\lambda}{n}\right)^{-1}
 (r_n(\widehat{\mathbf T}_{\widehat k}^{\epsilon}\mathbf\Lambda) - r_n(\mathbf T^0\mathbf\Lambda) + 4x_{\widehat k,\epsilon}) +
 \mathfrak c_2\epsilon
 \nonumber\\
 & \leqslant &
 \left(1 -\mathfrak c_{\ref{preliminary_risk_bound}}\frac{\lambda}{n}\right)^{-1}
 (r_n(\widehat{\mathbf T}_{\widehat k}\mathbf\Lambda) - r_n(\mathbf T^0\mathbf\Lambda) +\mathfrak c_1(\xi_1,\dots,\xi_n)\epsilon + 4x_{\widehat k,\epsilon}) +
 \mathfrak c_2\epsilon
 \nonumber\\
 & = &
 \left(1 -\mathfrak c_{\ref{preliminary_risk_bound}}\frac{\lambda}{n}\right)^{-1}
 \nonumber\\
 & &
 \times\left(
 \min_{k\in\mathcal K}
 \{r_n(\widehat{\mathbf T}_k\mathbf\Lambda) - r_n(\mathbf T^0\mathbf\Lambda) +\textrm{pen}(k)\} +\mathfrak c_1(\xi_1,\dots,\xi_n)\epsilon + 4x_{\widehat k,\epsilon} -\textrm{pen}(\widehat k)\right)
 +\mathfrak c_2\epsilon
 \nonumber\\
 & \leqslant &
 \frac{1}{1 -\mathfrak c_{\ref{preliminary_risk_bound}}\lambda n^{-1}}
 \min_{k\in\mathcal K}\{(1 +\mathfrak c_{\ref{preliminary_risk_bound}}\lambda n^{-1})
 \|(\widehat{\mathbf T}_k -\mathbf T^0)\mathbf\Lambda\|_{\mathcal F,\Pi}^{2} + 4x_{k,\epsilon} +\textrm{pen}(k)\}
 \nonumber\\
 & &
 \hspace{8cm} +
 \frac{\mathfrak m_n +\mathfrak c_1(\xi_1,\dots,\xi_n)\epsilon}{1 -\mathfrak c_{\ref{preliminary_risk_bound}}\lambda n^{-1}} +\mathfrak c_2\epsilon
 \nonumber\\
 \label{risk_bound_adaptive_estimator_2}
 & \leqslant &
 2\min_{k\in\mathcal K}\{
 3/2\|(\widehat{\mathbf T}_k -\mathbf T^0)\mathbf\Lambda\|_{\mathcal F,\Pi}^{2} + 2\textrm{pen}(k)\} +
 4\mathfrak m_n + (2\mathfrak c_1(\xi_1,\dots,\xi_n) +\mathfrak c_2)\epsilon
\end{eqnarray}
with
\begin{displaymath}
\mathfrak c_1(\xi_1,\dots,\xi_n) :=
2\mathfrak m_{\mathbf\Lambda}
\left(\frac{1}{n}\sum_{i = 1}^{n}|\xi_i| +\mathfrak m_{\varepsilon} + 2\mathfrak m_0\right)
\textrm{ and }
\mathfrak c_2 = 4\mathfrak m_0\mathfrak m_{\mathbf\Lambda}.
\end{displaymath}
Moreover, by following the proof of Theorem \ref{risk_bound} and Theorem \ref{explicit_risk_bound} on the same event $\Omega_{\lambda,\epsilon}$,
\begin{displaymath}
\|(\widehat{\mathbf T}_k -\mathbf T^0)\mathbf\Lambda\|_{\mathcal F,\Pi}^{2}
\leqslant
3\min_{\mathbf T\in\mathcal S_k}
\|(\mathbf T -\mathbf T^0)\mathbf\Lambda\|_{\mathcal F,\Pi}^{2}
+\mathfrak c_{\ref{explicit_risk_bound}}\left[k(d +\tau)\frac{\log(n)}{n}
+\frac{1}{n}\log\left(\frac{2}{\alpha}|\mathcal K|\right)\right]
\end{displaymath}
for every $k\in\mathcal K$. Therefore, thanks to (\ref{risk_bound_6}), (\ref{risk_bound_7}) and (\ref{risk_bound_adaptive_estimator_2}), with probability larger than $1 - 2\alpha$,
\begin{eqnarray*}
 \|\widehat{\mathbf\Theta} -\mathbf\Theta^0\|_{\mathcal F,\Pi}^{2}
 & \leqslant &
 4\min_{k\in\mathcal K}\left\{
 3\min_{\mathbf T\in\mathcal S_k}
 \|(\mathbf T -\mathbf T^0)\mathbf\Lambda\|_{\mathcal F,\Pi}^2 +
 (\mathfrak c_{\ref{explicit_risk_bound}} + 16\mathfrak c_{{\rm pen}})k(d +\tau)\frac{\log(n)}{n}
 \right\}\\
 & & +
 \frac{4\mathfrak c_{\ref{explicit_risk_bound}} + 16\mathfrak c_{{\rm pen}}}{n}\log\left(\frac{2}{\alpha}|\mathcal K|\right)
 + 9\mathfrak m_0
 \frac{d^{1/2}\tau^{1/2}}{\mathfrak m_{\mathbf\Lambda}(\tau)n^2}
 \left[\mathfrak c_{\ref{risk_bound},2} +
 8\mathfrak m_{\bf\Lambda}\mathfrak c_{\xi}\log\left(\frac{1}{\alpha}\right)\right].
\end{eqnarray*}
To end the proof, let us replace $\alpha$ by $\alpha/2$ and note that $d^{1/2}\tau^{1/2}/(\mathfrak m_{\bf\Lambda}(\tau)n^2)\leqslant 1/n$ because $n\geqslant\max(d,\tau)$ and $\mathfrak m_{\bf\Lambda}(\tau)\geqslant 1$.
\\
\\
{\bf Acknowledgements.} This work was partially funded by CY Initiative of Excellence (grant "Investissements d'Avenir" ANR-16-IDEX-0008), Project "EcoDep" PSI-AAP2020-0000000013.
%


%

%
\end{document}